\documentclass[10pt,reqno]{amsart} 

\usepackage{dsfont}
\usepackage{hyperref}
\usepackage{amsthm,amsmath,amsfonts,amssymb,xcolor,hyperref,cases,tikz, mathtools}
\usepackage{hyperref}
\hypersetup{linkcolor=blue, colorlinks=true,citecolor = red}
\usepackage[capitalise]{cleveref}
\usepackage{vmargin}
\setmarginsrb{2cm}{1cm}{2cm}{3cm}{1cm}{1cm}{2cm}{2cm}

\everymath{\displaystyle}

\numberwithin{equation}{section}
\newtheorem{Theorem}{Theorem}[section]

\newtheorem{Definition}[Theorem]{Definition}
\newtheorem{Remark}[Theorem]{Remark}
\renewcommand{\leq}{\leqslant}
\renewcommand{\le}{\leqslant}
\renewcommand{\geq}{\geqslant}

\def\hmath$#1${\texorpdfstring{{\rmfamily\textit{#1}}}{#1}}

\newcommand{\mc}{\mathcal}
\newcommand{\vu}{\mathbf{u}}
\title[]{\sc Compressible Navier--Stokes system with the hard sphere pressure law in an exterior domain}

\author{}

\author{\v S. Ne\v casov\'  a $^*$}
\thanks{$^*$ the corresponding author}
\address{\v S. Ne\v casov\' a 
	\newline \indent
	{Institute of Mathematics of the Czech Academy of Sciences, \newline \indent
		\v Zitn\'a 25, 115 67 Praha 1, Czech Republic.}}
		\email{matus@math.cas.cz}
\author{A. Novotn\' y}	
\address{A. Novotn\' y
	\newline \indent
	{IMATH, EA 2134, Universit\'e de Toulon Var, BP20132, 83957 La Garde, France.}}
		\email{novotny@univ-tln.fr}\thanks{}
		
		\author{Arnab Roy}
\address{Arnab Roy 
	\newline \indent
	{Institute of Mathematics of the Czech Academy of Sciences, \newline \indent
		\v Zitn\'a 25, 115 67 Praha 1, Czech Republic.}}
		\email{royarnab244@gmail.com}	

\date{\today}

\keywords{Compressible Navier-Stokes system, Hard sphere pressure law, Exterior domain}
\subjclass[2010]{}

\begin{document}

\begin{abstract}
We consider the motion of compressible Navier-Stokes fluids with the hard sphere pressure law around a rigid obstacle when the velocity and the density at infinity are non zero. This kind of pressure model is largely employed in various physical and industrial applications. We prove the existence of weak solution to the system in the exterior domain.
\end{abstract}

\maketitle
{\bf The article was finished shortly after death of A. Novotn\' y. We never forget him.}

\tableofcontents

\section{Introduction} \label{intro}
We consider a bounded domain $\mc{ S}\subset \mathbb{R}^d$, $d=2,3$ of class $C^2$ with
boundary $\partial\mc{S}$. Let us denote the open ball with radius $R$ with the center at the origin by $B_R$ and  without loss of generality let us assume that $\overline{\mc{S}}\subset B_{1/2}$. Let $\Omega$ be an exterior domain given by
\begin{equation*}
\Omega  := \mathbb{R}^d \setminus \overline{\mc{S}},\quad d=2,3. 
\end{equation*}
 We consider the motion of viscous compressible fluid in the exterior domain $\Omega$ around  the obstacle $\mathcal S$. Precisely, the mass density $\rho=\rho(t,x)$ and the velocity $\mathbf{u}=\mathbf{u}(t,x)$ of the fluid satisfy the following system:
\begin{align}
{\partial _t \rho }+ \mbox {div }(\rho \mathbf u)  & = 0 \quad \mbox{in}\quad (0,T)\times {\Omega}\label{eq:mass}\\
\partial _{t}(\rho\mathbf u)+\mbox { div } \left( \rho\mathbf u \otimes \mathbf u\right) +\nabla p(\rho)-
\mbox { div } {\mathbb  S(\nabla \mathbf u)} & =  0
\quad \mbox{in}\quad (0,T)\times {\Omega}\label{eq:mom}\\
 \mathbb {S}(\nabla \vu) &=
\mu(\nabla \vu + \nabla^{\top} \vu)  + \lambda\operatorname{div}\vu \mathbb{I},\quad
 \mu>0, \lambda\geq 0,\label{eq:tensor}\\
\mathbf u & =  0
\quad \mbox{on}\quad (0,T)\times \partial{\mc{S}},\label{eq:bdary}
\end{align}
where $p=p(\rho)$ is the hard sphere pressure. The system is
endowed with the initial conditions
\begin{equation}\label{ini}
\rho(0,x)=\rho_0(x),\quad \rho\vu (0,x)=\mathbf{q}_0 (x),\quad x\in\Omega .
\end{equation}
Since $\Omega$ is an exterior domain, we need to prescribe the behaviour of $(\rho,\mathbf{u})$ at infinity:
\begin{equation}\label{behaviour}
\rho(t,x)  \rightarrow  \rho_{\infty}, \quad \mathbf u(t,x) \rightarrow \mathbf a_{\infty}\quad \mbox{as} \quad |x|\rightarrow\infty,\ (t,x)\in (0,T)\times\Omega,
\end{equation}
where $\mathbf a_{\infty}\in \mathbb R^3$ is a nonzero constant vector and $\rho_{\infty}>0$ is a given positive constant. We assume that the fluid density cannot exceed a limit value $\overline{\rho}>0$ and the hard sphere pressure $p=p(\rho)$, as a function of the density, becomes infinite when the density approaches a finite critical value $\overline{\rho}>0$:
\begin{equation*}
\lim_{\rho\rightarrow \overline{\rho}} p(\rho)=\infty.
\end{equation*}
The above condition of hard sphere model eliminates the possibility of the standard pressure law for the isentropic gases. Specifically, we consider the pressure $p\in C^1[0,\overline{\rho})\cap C^2(0,\overline{\rho})$ satisfies
\begin{equation}\label{p-law}
p(0)=0,\quad p'(\rho)>0\ \forall\ \rho>0,\quad \liminf_{\rho\rightarrow 0}\frac{p'(\rho)}{\rho}>0,\quad p(\rho)\sim _{\rho\rightarrow \overline{\rho}-} |\overline{\rho}-\rho|^{-\beta},\mbox{ for some }\beta > 5/2,
\end{equation}
where without loss of generality, we assume that $1< \overline{\rho}<\infty$ and $a(s)\sim_{s\rightarrow s_0\pm} b(s)$ stands for 
\begin{equation*}
c_1a(s)\leq b(s) \leq c_2 a(s),\mbox{ in a right}(+),\mbox{ left}(-)\mbox{ neighbourhood of }s_0.
\end{equation*}

We study the well-accepted Carnahan-Starling equation of the state (\ref{p-law}). It
is an approximate but quite a good (as explained in \cite{Song}) equation of state for the fluid phase of the hard sphere model. Such model was derived from a quadratic relation between the integer portions of the virial coefficients and their orders. This model is convenient for the initial study of the behavior of dense gases and liquids.  For more details regarding this model and several corrections (Percus-Yevick equation, Kolafa correction, Liu correction), we refer to \cite{Carnahan-Starling, Hongqin,KLM04AEHS,KVB62CPES}.
Similar type of the singular pressure law are considered in many physical models. Let us mention the work by Degond and Hu \cite{MR3020033} and Degond et all \cite{MR2835410} for collective motion. Moreover, the work of Berthelin et al. \cite{MR2438216, MR2366138} for the trafic flow, the paper by Maury \cite{maury2012}
	  concerning the crowd motion models. Similar type of models can be found also in works of Bresch et al. \cite{bresch2014, bresch2017}, Perrin and Zatorska \cite{peza2015}, Bresch, Ne\v casov\' a and Perrin \cite{MR3974475}.

The existence of weak solutions in the case of the barotropic situation is going back to the 
seminal work by Lions \cite{MR1637634} and improved by Feireisl et al. \cite{MR1867887}. The question about the existence of weak solutions of the  hard pressure case in a bounded domain with no-slip boundary conditions were studied recently by Feireisl and Zhang \cite[Section 3]{MR2646821}. The case of  general inflow/outflow was investigated  by Choe, Novotn\' y and Yang \cite{MR3912678}. Weak--strong uniqueness in the case of the hard pressure in periodic spatial domains was shown by Feireisl, Lu and Novotn\' y \cite{FLN}. In our knowledge, there is no available existence result of weak solutions with hard sphere pressure law in the case of an unbounded domain. In the case of barotropic compressible fluid,
the existence of weak solutions in an unbounded domain when the velocity at infinity and the density at infinity are nonzero has been done by Novotn\' y, Stra\v skraba \cite[ Section 7.12.6]{MR2084891} and by Lions \cite[Section 7]{MR1637634}. The case of the motion of the compressible fluids in $\mathbb{R}^3$ around a rotating obstacle where the velocity at infinity is nonzero and parallel to the axis of rotation was shown in the paper by Kra\v cmar, Ne\v casov\' a and Novotn\' y \cite {MR3208793}.

In this work, our aim is to establish the existence of weak solutions to the compressible Navier--Stokes system with hard pressure in the context of exterior domain. The main idea is to use the method of ``invading domains''. The exterior domain $\Omega$ is approximated by invading
domains $\Omega_R=\Omega \cap B_R$ and to begin with, we have to show the existence of solution in these bounded domains. Then we need to find estimates independent of  domains $\Omega_R$ so that we can identify the weak limits of the growing invading domains (as radius $R$ of $B_R$ goes to infinity) via
local weak compactness results \cite[Lemma 6.6]{MR2084891}. We have to use div-curl lemma, effective viscous flux, commutator lemmas, renormalized solutions of the transport equation frequently in our analysis. The complete methodology has been explained in Novo \cite{MR2189672} and in Novotn\' y, Stra\v skraba \cite[ Section 7.12.6]{MR2084891} for the case of compressible barotropic fluid and we have adapted it in this paper for compressible fluid with hard sphere pressure law.

%show the existence of solution in a bounded domain (let us say at ball with radius $R$) and using the invading domain and  by the construction that the boundary condition are admissibly prolongated to $\mathbb{R}^3$ \footnote{Here we can follow the work by Novo \cite{MR2189672}} and then pass to the limit with the radius then the existence of weak solutions is proven.  

The outline of the paper is as follows. Section \ref{intro} deals with the description of the problem, the meaning of weak solution to the problem and the statement of the main result of the paper. The approximation problem on large balls via a suitable penalization is introduced in Section \ref{Approx}. The penalization uses an auxiliary vector field $\mathbf{u}_{\infty}$ (defined in \eqref{d2}) which is very crucial to achieve the required behaviour of velocity at infinity in the limiting process. Moreover, the existence of such problem is shown and the limit process is performed with the penalization parameter tending to infinity where equi-integrability of the pressure is important to pass the limit in the pressure term. 
Section \ref{proof} is devoted to the proof of \cref{thm:main} where the limit with the radius of large balls tending to infinity is achieved and the method of ``invading domains'' is used. In this step, the special choice of the test functions is crucial to identify the limit of the pressure.

%%%%%%%%%%%%%%%%%%%%%%%%%%%%%%%%%%%%%%%%%%%%%%

\subsection{Weak formulation and Main result}\label{weak}
We want to define the notion of weak solutions to system \eqref{eq:mass}--\eqref{behaviour} together with the pressure $p(\rho)$ satisfying \eqref{p-law}. Let us denote the open ball with radius $R$ with the center at the origin by $B_R.$
To start with,
without loss of generality, assume that $
\overline  { \mathcal S}\subset B_{1/2}$
 and $B_1\subset
B_R$. We set 
\begin{align}\label{d1}
\mathbf{U}_{\infty} := \begin{cases}
\vu_{\infty}&\mbox{ in }B_1,\\
\mathbf{a}_{\infty}&\mbox{ in }\mathbb{R}^3\setminus B_1,
\end{cases}
\end{align}
where $\vu_{\infty}\in C^{1}_c(\mathbb R^3)$ is such that:
\begin{equation}\label{d2}
\vu_{\infty}= \left \{ \begin{array}{l}
0 \mbox { in } B_{1/2},\\
\mathbf{a}_{\infty } \mbox { in } B_{(3/2)R}\setminus B_1,
\\
\mbox { supp } \vu_{\infty } \subset B_{2R}.
\end{array}
\right. \quad \mbox { div } \vu_{\infty} = 0 \mbox { in }\mathbb R^3.
\end{equation}
The construction of such vector field $\vu_{\infty}$ follows from the explanations \cite[Section 3, page 195]{MR3208793}, \cite[Section 1, page 487--488]{MR2189672}  and the result \cite[Exercise III.3.5, page 176]{MR2808162}.
\begin{Definition}\label{def:bddenergy}
 We say that a couple $(\rho,\vu)$  is a bounded energy weak solution of the problem  \eqref{eq:mass}--\eqref{behaviour} with the pressure law \eqref{p-law} if the following conditions are satisfied:
\begin{itemize}
\item Functions $(\rho,\vu)$ are such that
\begin{equation*}
0\leq \rho < \overline{\rho},\quad E(\rho|\rho_{\infty})\in
L^\infty(0,T;L^1(\Omega)),
\end{equation*}
\begin{equation*}
\rho|\vu-\mathbf{U}_\infty|^2\in L^\infty(0,T;L^1(\Omega)),\quad (\vu-
\mathbf{U}_\infty)\in L^2(0,T;W^{1,2}(\Omega)).
\end{equation*}
In the above
\begin{equation*}
E(\rho|\rho_{\infty}) =
H(\rho)-H'(\rho_{\infty})(\rho-\rho_{\infty})-H(\rho_{\infty}),
\end{equation*}
where
\begin{equation}\label{def:H}
H(\rho) = \rho \int _1^{\rho} \frac{p(s)}{s^2}ds,
\end{equation}

\item The function $\rho\in C_{\rm{weak}}([0,T]; L^1(K))$ for any compact $K\subset\overline\Omega$ and the
equation of continuity \eqref{eq:mass} is satisfied in the weak sense,
\begin{equation}\label{eqf:weakmass}
 \int_\Omega \rho(\tau,\cdot)\varphi(\tau,\cdot)\ {\rm
d}x - \int_\Omega \rho_0(\cdot)\varphi(0,\cdot)\ {\rm
d}x=\int_0^\tau \int_{\Omega}\left(\rho
\partial_t \varphi + \rho\vu \cdot \nabla \varphi\right)
 \ dx\ dt,
\end{equation}
for all $\tau\in [0,T]$ and any test function $\varphi
\in C^1_{c}([0,T] \times \overline{\Omega})$.
\item The linear momentum $\rho\vu\in C_{\rm weak}([0,T], L^{1}(K))$ for
any compact $K\subset\overline \Omega$ and the momentum equation \eqref{eq:mom}
is satisfied in the weak sense
\begin{multline}\label{eqf:weakmom}
\int_\Omega\rho\vu(\tau,\cdot)\cdot\varphi(\tau,\cdot){\rm d} x - \int_\Omega \mathbf{q}_0(\cdot)\cdot\varphi(0,\cdot){\rm d} x  \\
 =\int_0^\tau \int_{\Omega}\Big(
\rho \vu \cdot \partial_t \varphi + \rho \vu
\otimes \vu : \nabla \varphi +
p(\rho)\operatorname{div}\varphi - \mathbb {S}(\nabla \vu) : \nabla \varphi \Big)\ dx \ dt,
\end{multline}
for all $\tau\in [0,T]$ and
 for any test
function $\varphi \in C^1_c([0,T] \times \Omega)$.
\item The following energy inequality holds: for a.e. $\tau\in (0,T)$,
\begin{multline}\label{eqf:ee}
\int_{\Omega}\Big(\frac{1}{2}\rho|\vu-\mathbf {U}_{\infty}|^2 +
E(\rho|\rho_{\infty})\Big)(\tau)\ dx
+ \int_0^\tau\int_{\Omega} \mathbb{S}(\nabla (\vu-\mathbf{U}_{\infty})):\nabla (\vu-\mathbf{U}_{\infty})\ dx\ dt\\
\leq
  \int_{\Omega}\Big(\frac{1}{2}\frac{|\mathbf{q}_0-\rho_{0}\mathbf{U}_{\infty}|^2}{\rho_0} +
E(\rho_0|\rho_{\infty})\Big)\ dx - \int_0^\tau
\int_{B_1\setminus\mc{S}} \rho\vu\cdot\nabla\mathbf {U}_{\infty}
\cdot(\vu - \mathbf{U}_{\infty})\,{ d}x\,{ d}t - \int _0^{ \tau}\int _{B_1\setminus\mc{S}}\mathbb{S}(\nabla\mathbf{U}_{\infty}):\nabla(\vu-\mathbf{U}_\infty)\,{ d} x\, { d}t.
\end{multline}
\end{itemize}
\end{Definition}
\begin{Remark}
 We can use the regularization procedure in the transport theory by DiPerna and Lions \cite{DiPerna1989} to show that if $(\rho,\vu)$ is a bounded energy weak solution of the problem  \eqref{eq:mass}--\eqref{behaviour} according to \cref{def:bddenergy}, then $(\rho,\vu)$ also satisfy a renormalized continuity equation in a weak sense, i.e,
 \begin{equation}\label{eq:renorm}
\partial_t b(\rho) + \operatorname{div}(b(\rho)\vu) + (b'(\rho)-b(\rho))\operatorname{div}\vu=0 \mbox{ in }\, \mc{D}'([0,T)\times {\Omega}) ,
\end{equation} 
for any $b\in C([0,\infty)) \cap C^1((0,\infty))$.
\end{Remark}
%\begin{Remark}
%\end{Remark}
We are now in a position to state the main result of the present paper.
\begin{Theorem}\label{thm:main}
Assume that
$0< \rho_{\infty} < \overline{\rho}$, $\mathbf{a}_{\infty}(\neq 0)\in \mathbb{R}^d$, $\mathbf{U}_{\infty}\in C_c^1(\mathbb{R}^d)$ is defined by \eqref{d1} and $\Omega =\mathbb R^d\setminus
\overline{\mathcal S}$, where $\mc{S}\subset \mathbb{R}^d$, $d = 2, 3$ is a bounded domain of class $C^2$. Assume that the pressure satisfies the hypothesis \eqref{p-law}, the initial data have finite energy
\begin{equation}\label{init}
 E(\rho_0|\rho_{\infty}) \in L^1(\Omega),\quad 0\leq \rho_{0} < \overline{\rho},\quad
\frac{|\mathbf{q}_0-\rho_{0}\mathbf{a}_{\infty}|^2}{\rho_0}\mathds{1}_{\{\rho_0 > 0\}}\in L^1(\Omega).
%\quad \mc{M}_0:=\frac{1}{|\Omega|}\int\limits_{\Omega} \rho_0 > 0.
\end{equation}
Then  the problem \eqref{eq:mass}--\eqref{behaviour}  admits at least one renormalized bounded energy weak solution on $(0,T)\times\Omega$.
\end{Theorem}

%%%%%%%%%%%%%%%%%%%%%%%%%%%%%%%%%%%%%%%%%%%

\section{Approximate problems in bounded domain}\label{Approx}
In this section, in order to solve system \eqref{eq:mass}--\eqref{behaviour}, we want to propose some approximate problems in a bounded domain and to analyze the well-posedness of such problems.
\subsection{Existence of a penalized problem}
Let us denote $V:=B_{2R}$. In order to construct solutions to \cref{thm:main}, we start with the following penalized problem:
\begin{align}
{\partial _t \rho }+ \mbox {div }(\rho \mathbf u)  & = 0 \quad \mbox{in}\quad (0,T)\times {V}\label{eq:penmass},\\
\partial _{t}(\rho\mathbf u)+\mbox { div } \left( \rho\mathbf u \otimes \mathbf u\right) +\nabla p(\rho)-
\mbox { div } {\mathbb  S(\nabla \mathbf u)} + m\mathds{1}_{\{(V\setminus B_R)\cup \mc{S}\}}(\vu-\vu_{\infty}) & =  0
\quad \mbox{in}\quad (0,T)\times {V}\label{eq:penmom},\\
\mathbf u & =  0
\quad \mbox{on}\quad (0,T)\times \partial{V},\label{eq:penbdary}
\end{align}
\begin{equation}\label{penini}
\rho(0,x)=\rho_0(x),\quad \rho\vu (0,x)=\mathbf{q}_0 (x),\quad x\in V .
\end{equation}
In the above, the initial data $\rho_0$ and $\vu_0$ have been extended by zero in $\mc{S}$. 
\begin{Definition}\label{def:penbddenergy}
 We say that a couple $(\rho_m,\vu_m)$  is a bounded energy weak solution of the problem  \eqref{eq:penmass}--\eqref{penini} with \eqref{p-law} if the following conditions are satisfied:
\begin{itemize}
\item Functions $(\rho_m,\vu_m)$ are such that
\begin{equation*}
0\leq \rho_m < \overline{\rho},\quad E(\rho_m|\rho_{\infty})\in
L^\infty(0,T;L^1(V)),
\end{equation*}
\begin{equation*}
\rho_m|\vu_m-\mathbf{u}_\infty|^2\in L^\infty(0,T;L^1(V)),\quad (\vu_m-
\mathbf{u}_\infty)\in L^2(0,T;W_0^{1,2}(V)).
\end{equation*}

\item The function $\rho_m\in C_{\mbox{weak}}([0,T]; L^1(V))$ and the
equation of continuity \eqref{eq:penmass} is satisfied in the weak sense,
\begin{equation}\label{eq:weakpenmass}
 \int_V \rho_m(\tau,\cdot)\varphi(\tau,\cdot)\ {\rm
d}x - \int_V \rho_0(\cdot)\varphi(0,\cdot)\ {\rm
d}x=\int_0^\tau \int_{V}\left(\rho_m
\partial_t \varphi + \rho_m\vu_m \cdot \nabla \varphi\right)
 \ dx\ dt,
\end{equation}
for all $\tau\in [0,T]$ and any test function $\varphi
\in C^1_{c}([0,T] \times \overline{V})$.
\item The linear momentum $\rho_m\vu_m\in C_{\rm weak}([0,T], L^{1}(V))$ and the momentum equation \eqref{eq:penmom}
is satisfied in the weak sense
\begin{multline}\label{eq:weakpenmom}
\int_V\rho_m\vu_m(\tau,\cdot)\cdot\varphi(\tau,\cdot){\rm d} x - \int_V \mathbf{q}_0(\cdot)\cdot\varphi(0,\cdot){\rm d} x  \\
 =\int_0^\tau \int_{V}\Big(
\rho_m \vu_m \cdot \partial_t \varphi + \rho_m \vu_m
\otimes \vu_m : \nabla \varphi +
p(\rho_m)\operatorname{div}\varphi - \mathbb {S}(\nabla \vu_m) : \nabla \varphi - m\mathds{1}_{\{(V\setminus B_R)\cup \mc{S}\}}(\vu_m-\vu_{\infty})\cdot \varphi \Big)\ dx \ dt,
\end{multline}
for all $\tau\in [0,T]$ and
 for any test
function $\varphi \in C^1_c([0,T] \times V)$.
\item The following energy inequality holds: for a.e. $\tau\in (0,T)$,
\begin{multline}\label{pen:energy}
\int_{V}\Big(\frac{1}{2}\rho_m|\vu_m-\mathbf {u}_{\infty}|^2 +
E(\rho_m|\rho_{\infty})\Big)(\tau)\ dx
+ \int_0^\tau\int_{V}\Big( \mathbb{S}(\nabla (\vu_m-\vu_{\infty})):\nabla (\vu_m-\vu_{\infty})  +m\mathds{1}_{\{(V\setminus B_R)\cup \mc{S}\}} |\vu_m-\vu_{\infty}|^2\Big)\ dx\ dt\\
\leq
  \int_{V}\Big(\frac{1}{2}\frac{|\mathbf{q}_0-\rho_{0}\mathbf{u}_{\infty}|^2}{\rho_0} +
E(\rho_0|\rho_{\infty})\Big)\ dx + \int_0^\tau
\int_{V} \rho_m\vu_m\cdot\nabla\mathbf {u}_{\infty}
\cdot(\mathbf{u}_{\infty}- \vu_m)\,{ d}x\,{ d}t - \int _0^{ \tau}\int _V\mathbb{S}(\nabla\vu_{\infty}):\nabla(\vu_m-\vu_\infty)\,{ d} x\, { d}t.
\end{multline}
\end{itemize}
\end{Definition}
\begin{Remark}
 We can use the regularization procedure in the transport theory by DiPerna and Lions \cite{DiPerna1989} to show that if $(\rho_m,\vu_m)$ is a bounded energy weak solution of the problem  \eqref{eq:penmass}--\eqref{penini} according to \cref{def:penbddenergy}, then $(\rho_m,\vu_m)$ also satisfy a renormalized continuity equation in a weak sense, i.e,
 \begin{equation}\label{eq:penrenorm}
\partial_t b(\rho_m) + \operatorname{div}(b(\rho_m)\vu_m) + (b'(\rho_m)-b(\rho_m))\operatorname{div}\vu_m=0 \mbox{ in }\, \mc{D}'([0,T)\times \overline{V}) ,
\end{equation} 
for any $b\in C([0,\infty)) \cap C^1((0,\infty))$.
\end{Remark}
We can prove an existence result to the penalized problem \eqref{eq:penmass}--\eqref{penini} by following the idea of Feireisl, Zhang \cite[Theorem 3.1]{MR2646821}:
\begin{Theorem}\label{thm:pen}
Assume that
$0<\rho_{\infty}<\overline{\rho}$, $\mathbf{u}_{\infty}\in C_c^1(\mathbb{R}^d)$ is defined by \eqref{d2}. Assume that the pressure satisfies the hypothesis \eqref{p-law}, the initial data satisfy
\begin{equation}\label{ini:m}
 E(\rho_0|\rho_{\infty}) \in L^1(V),\quad 0\leq \rho_{0} < \overline{\rho},\quad
\frac{|\mathbf{q}_0-\rho_{0}\mathbf{a}_{\infty}|^2}{\rho_0}\mathds{1}_{\{\rho_0 > 0\}}\in L^1(V).
%\quad \mc{M}_{0,V}:=\frac{1}{|V|}\int\limits_{V} \rho_0 > 0.
\end{equation}
Then  the problem \eqref{eq:penmass}--\eqref{penini}  admits at least one renormalized bounded energy weak solution $(\rho_m,\vu_m)$ on $(0,T)\times V$.
\end{Theorem}
\begin{proof}
We consider a family of solutions $(\rho_{\varepsilon},\vu_{\varepsilon})$ of an approximate problem with regularized pressure $p_{\varepsilon}(\rho)$:
\begin{align}
{\partial _t \rho_{\varepsilon} }+ \mbox {div }(\rho_{\varepsilon} \mathbf u_{\varepsilon})  & = 0 \quad \mbox{in}\quad (0,T)\times {V}\label{eq:epmass},\\
\partial _{t}(\rho_{\varepsilon}\mathbf u_{\varepsilon})+\mbox { div } \left( \rho_{\varepsilon}\mathbf u_{\varepsilon} \otimes \mathbf u_{\varepsilon}\right) +\nabla p_{\varepsilon}(\rho_{\varepsilon})-
\mbox { div } {\mathbb  S(\nabla \mathbf u_{\varepsilon})} + m\mathds{1}_{\{(V\setminus B_R)\cup \mc{S}\}}(\vu_{\varepsilon}-\vu_{\infty}) & =  0
\quad \mbox{in}\quad (0,T)\times {V}\label{eq:epmom},\\
\mathbf u_{\varepsilon} & =  0
\quad \mbox{on}\quad (0,T)\times \partial{V},\label{eq:epbdary}
\end{align}
\begin{equation}\label{epini}
\rho_{\varepsilon}(0,x)=\rho_0(x),\quad \rho_{\varepsilon}\vu_{\varepsilon} (0,x)=(\rho\vu)_0 (x),\quad x\in V ,
\end{equation}
where regularized pressure $p_{\varepsilon}$ is given by 
\begin{align}\label{reg:pre}
p_{\varepsilon}(\rho)=\begin{cases}
p(\rho)&\mbox{ for }\rho \in [0,\overline{\rho}-\varepsilon]\\
p(\overline{\rho}-\varepsilon) + |(\rho - \overline{\rho}+\varepsilon)^{+}|^{\gamma}&\mbox{ for }\rho \in (\overline{\rho}-\varepsilon,\infty),
\end{cases}
\end{align}
for a certain exponent $\gamma>d$ (which is chosen sufficiently large).
The idea is to establish the existence of the problem \eqref{eq:penmass}--\eqref{penini} as an asymptotic limit of the family $(\rho_{\varepsilon},\vu_{\varepsilon})$ as $\varepsilon\rightarrow 0$. Under the assumptions \eqref{init} on initial data and $\vu_{\infty}$ \eqref{d2}, the problem \eqref{eq:epmass}--\eqref{epini} with regularized pressure law \eqref{reg:pre} admits at least one weak solution $(\rho_{\varepsilon},\vu_{\varepsilon})$ by following the idea of \cite[Theorem 7.79, Page 425]{MR2084891}. Then we can follow \cite[Section 3]{MR2646821} to obtain the uniform bounds (with respect to $\varepsilon$) of the density $\{{\rho}_{\varepsilon}\}$, velocity $\{\vu_{\varepsilon}\}$, the pressure $\{p_{\varepsilon}\}$ and the equi-integrability of the pressure family so that we can pass the limit $\varepsilon \rightarrow 0$ and most importantly, we can conclude 
\begin{equation*}
p_{\varepsilon}({\rho}_{\varepsilon})\rightarrow p(\rho)\mbox{ in }L^1((0,T)\times V).
\end{equation*}
\end{proof}
%%%%%%%%%%%%%%%%%%%%%%%%%%%%%%%%%%%%%%%%%%%

 \subsection{Limit \texorpdfstring{$m\rightarrow \infty$}{}}
 
 Let us denote $\Omega_R:= B_R \setminus \overline{\mc{S}}$ and we consider the following system:
 \begin{align}
{\partial _t \rho }+ \mbox {div }(\rho \mathbf u)  & = 0 \quad \mbox{in}\quad (0,T)\times {V}\label{eq:mmass},\\
\partial _{t}(\rho\mathbf u)+\mbox { div } \left( \rho\mathbf u \otimes \mathbf u\right) +\nabla p(\rho)-
\mbox { div } {\mathbb  S(\nabla \mathbf u)} & =  0
\quad \mbox{in}\quad (0,T)\times \Omega_R \label{eq:mmom},\\
\mathbf u & =  0
\quad \mbox{on}\quad (0,T)\times \partial\mc{S},\label{eq:mbdary}\\
\vu &=\vu_{\infty} \mbox{ a.e. in }(0,T)\times [(V\setminus B_R)\cap \mc{S}],
\end{align}
\begin{equation}\label{mini}
\rho(0,x)=\rho_0(x),\quad \rho\vu (0,x)=\mathbf{q}_0 (x),\quad x\in V .
\end{equation}
 \begin{Definition}\label{def:mbddenergy}
 We say that a couple $(\rho_R,\vu_R)$  is a bounded energy weak solution of the problem  \eqref{eq:mmass}--\eqref{mini} with \eqref{p-law} if the following conditions are satisfied:
\begin{itemize}
\item Functions $(\rho_R,\vu_R)$ are such that
\begin{equation*}
0\leq \rho_R < \overline{\rho},\quad E(\rho_R|\rho_{\infty})\in
L^\infty(0,T;L^1(\Omega_R)),
\end{equation*}
\begin{equation*}
\rho_R|\vu_R-\mathbf{u}_\infty|^2\in L^\infty(0,T;L^1(\Omega_R)),\quad (\vu_R-
\mathbf{u}_\infty)\in L^2(0,T;W_0^{1,2}(\Omega_R)).
\end{equation*}

\item The function $\rho_R\in C_{\mbox{weak}}([0,T]; L^1(V))$ and the
equation of continuity \eqref{eq:mmass} is satisfied in the weak sense,
\begin{equation}\label{eq:weakmmass}
 \int_V \rho_R(\tau,\cdot)\varphi(\tau,\cdot)\ {\rm
d}x - \int_V \rho_0(\cdot)\varphi(0,\cdot)\ {\rm
d}x=\int_0^\tau \int_{V}\left(\rho_R
\partial_t \varphi + \rho_R\vu_R \cdot \nabla \varphi\right)
 \ dx\ dt,
\end{equation}
for all $\tau\in [0,T]$ and any test function $\varphi
\in C^1_{c}([0,T] \times \overline{V})$.
\item The linear momentum $\rho_R\vu_R\in C_{\rm weak}([0,T], L^{2}(V))$ and the momentum equation \eqref{eq:mmom}
is satisfied in the weak sense
\begin{multline}\label{eq:weakmmom}
\int_{\Omega_R}\rho_R\vu_R(\tau,\cdot)\cdot\varphi(\tau,\cdot){\rm d} x - \int_{\Omega_R} \mathbf{q}_0(\cdot)\cdot\varphi(0,\cdot){\rm d} x  \\
 =\int_0^\tau \int_{\Omega_R}\Big(
\rho_R \vu_R \cdot \partial_t \varphi + \rho_R \vu_R
\otimes \vu_R : \nabla \varphi +
p(\rho_R)\operatorname{div}\varphi - \mathbb {S}(\nabla \vu_R) : \nabla \varphi  \Big)\ dx \ dt,
\end{multline}
for all $\tau\in [0,T]$ and
 for any test
function $\varphi \in C^1_c([0,T] \times \Omega_R)$. Moreover, 
\begin{equation}\label{umequinf}
\vu_R=\vu_{\infty} \mbox{ a.e. in }(0,T)\times [(V\setminus B_R)\cap \mc{S}].
\end{equation}
\item The following energy inequality holds: for a.e. $\tau\in (0,T)$,
\begin{multline}\label{m:energy}
\int_{\Omega_R}\Big(\frac{1}{2}\rho_R|\vu_R-\mathbf {u}_{\infty}|^2 +
E(\rho_R|\rho_{\infty})\Big)(\tau)\ dx
+ \int_0^\tau\int_{\Omega_R} \mathbb{S}(\nabla (\vu_R-\vu_{\infty})):\nabla (\vu_R-\vu_{\infty})\ dx\ dt\\
\leq
  \int_{\Omega_R}\Big(\frac{1}{2}\frac{|\mathbf{q}_0-\rho_{0}\mathbf{u}_{\infty}|^2}{\rho_0} +
E(\rho_0|\rho_{\infty})\Big)\ dx + \int_0^\tau
\int_{B_1\setminus\mc{S}} \rho\vu\cdot\nabla\mathbf {u}_{\infty}
\cdot(\mathbf{u}_{\infty}- \vu)\,{ d}x\,{ d}t - \int _0^{ \tau}\int _{B_1\setminus\mc{S}}\mathbb{S}(\nabla\vu_{\infty}):\nabla(\vu-\vu_\infty)\,{ d} x\, { d}t.
\end{multline}
\end{itemize}
\end{Definition}
 We want to show the existence of the solution $(\rho_R,\vu_R)$ according to \cref{def:mbddenergy} to the system \eqref{eq:mmass}--\eqref{mini}  as a limit of the solution $(\rho_m,\vu_m)$ to the system   \eqref{eq:penmass}--\eqref{penini} as $m\rightarrow \infty$.
\begin{Theorem}\label{thm:m}
Assume that
$0<\rho_{\infty}<\overline{\rho}$, $\mathbf{u}_{\infty}\in C_c^1(\mathbb{R}^d)$ is defined by \eqref{d2}. Assume that the pressure satisfies the hypothesis \eqref{p-law}, the initial data satisfy
\begin{equation}\label{ini:R}
 E(\rho_0|\rho_{\infty}) \in L^1(\Omega_R),\quad 0\leq \rho_{0} < \overline{\rho},\quad
\frac{|\mathbf{q}_0-\rho_{0}\mathbf{u}_{\infty}|^2}{\rho_0}\mathds{1}_{\{\rho_0 > 0\}}\in L^1(\Omega_R).
%\quad \mc{M}_{0,\Omega_R}:=\frac{1}{|\Omega_R|}\int\limits_{\Omega_R} \rho_0 > 0.
\end{equation}
Then  the problem \eqref{eq:mmass}--\eqref{mini}  admits at least one renormalized bounded energy weak solution $(\rho_R,\vu_R)$ according to \cref{def:mbddenergy}.
\end{Theorem}

\begin{proof}
We denote by $c=c(R)$, a generic constant that may depend on $R$
but is independent of $m$. As $(\rho_m,\vu_m)$ satisfies energy inequality \eqref{pen:energy}, we can deduce the following estimates: \begin{equation}\label{E1}
\displaystyle 
\sup_{t\in (0,T)}\int _V \rho _m|\mathbf u_m-\mathbf u_\infty|^2
\mbox{d}x \leq c(R),
\end{equation}
\begin{equation}\label{E2}
\sup_{t\in (0,T)}\int _V E(\rho _m|\rho_{\infty})
\mbox{d}x \leq c(R), \end{equation}
\begin{equation}\label{E3}
\|\vu_m -\vu_{\infty}\|_{L^2((0,T)\times((V\setminus B_R)\cup\mathcal S ))}
\le \frac{c(R)}{\sqrt{m}},
\end{equation}
\begin{equation}\label{E4}
\|\vu_m \|_{L^2(0,T;W^{1,2}(V))}\le c(R),
\end{equation}
 where the estimates \eqref{E1}--\eqref{E3} are direct consequence of \eqref{pen:energy} and to derive \eqref{E4}, we need to use Korn-Poincar\'{e} type inequality along with \eqref{pen:energy}. Moreover, by virtue of \eqref{E1}, \eqref{E4} and boundedness of $\rho_m$ in $(0,T)\times V$, we have
 \begin{equation}\label{bd:rhou}
 \|\rho_m \vu_m\|_{L^{\infty}(0,T;L^{2}(V))} + \|\rho_m \vu_m\|_{L^2(0,T; L^{\frac{6q}{6+q}}(V))} \leq c(R),\mbox{  for any  }1\leq q <\infty.
 \end{equation}
\begin{equation}\label{bd:rhouu}
 \|\rho_m |\vu_m|^2\|_{L^2(0,T; L^{\frac{3}{2}}(V))} \leq c(R).
 \end{equation}

 \underline{Step 1: Limit in the continuity equation and boundedness of the density.}

 It follows from \cite[Section 3]{MR2646821} that
 \begin{equation}\label{rhoCw}
 \rho_m \mbox{ is bounded in }(0,T)\times V,\quad 
 \rho_m\in C_{\mbox{weak}}([0,T]; L^q(V))\mbox{  for any  }1\leq q <\infty,
 \end{equation}
 and
 \begin{equation}\label{rhouCw}
 \rho_m \vu_m\in C_{\mbox{weak}}([0,T]; L^2(V)).
  \end{equation}
 Using the renormalized continuity equation \eqref{eq:penrenorm}, we conclude that 
  \begin{equation}\label{rhoC}
 \rho_m\in C([0,T]; L^q(V))\mbox{  for any  }1\leq q <\infty.
 \end{equation}
 Consequently, we infer from relations \eqref{rhoC} and \eqref{E4} that
 \begin{align}
 \rho_m\rightarrow \rho_R &\mbox{  weakly--}* \mbox{ in }L^{\infty}((0,T); L^q(V)),\label{rhom}\\
 \vu_m \rightarrow \vu_R &\mbox{  weakly in   }L^2(0,T;W^{1,2}(V))\label{um}.
 \end{align}
 We obtain from the continuity equation \eqref{eq:weakpenmass} and the estimate \eqref{bd:rhou} that $\{\rho_m\}$ is uniformly continuous in $W^{-1,2}(V)$ on $[0,T]$. Since, it is also uniformly bounded in $L^q(V)$, we can apply Arzela-Ascoli \cite[Lemma 6.2, page 301]{MR2084891} to conclude 
  \begin{equation}\label{m:rhoCw}
 \rho_m \rightarrow \rho_R \mbox{ in } C_{\mbox{weak}}([0,T]; L^q(V))\mbox{  for any  }1\leq q <\infty.
 \end{equation}
 Furthermore, due to the compact imbedding  $L^2(V)\hookrightarrow \hookrightarrow W^{-1,2}(V)$, we obtain 
  \begin{equation}\label{m:rhos}
 \rho_m \rightarrow \rho_R \mbox{ strongly in } L^2(0,T; W^{-1,2}(V)).
 \end{equation}
 Moreover, the bound 
\eqref{bd:rhou} and the convergences \eqref{rhom}--\eqref{m:rhos} imply
\begin{equation}\label{con:rhou}
 \rho_m \vu_m \rightarrow \rho_R\vu_R \mbox{ weakly in }L^2(0,T; L^{\frac{6q}{6+q}}(V))  \mbox{  and   weakly}-* \mbox{ in }L^{\infty}(0,T;L^{2}(V)).
  \end{equation}
 Consequently, the convergences \eqref{rhom}--\eqref{con:rhou} enable us to pass the limit $m\rightarrow \infty$ in the equation \eqref{eq:weakpenmass} and we obtain:
 \begin{equation*}
 \int_V \rho_R(\tau,\cdot)\varphi(\tau,\cdot)\ {\rm
d}x - \int_V \rho_0(\cdot)\varphi(0,\cdot)\ {\rm
d}x=\int_0^\tau \int_{V}\left(\rho_R
\partial_t \varphi + \rho_R\vu_R \cdot \nabla \varphi\right)
 \ dx\ dt,
\end{equation*}
for all $\tau\in [0,T]$ and any test function $\varphi
\in C^1_{c}([0,T] \times \overline{V})$. 

Now we need to prove that $\rho_R$ is bounded. We already know from \cref{thm:pen} that
\begin{equation}\label{bdd:rhom}
0\leq \rho_m < \overline{\rho}\quad\mbox{ in }\quad (0,T)\times V.
\end{equation}
Moreover, we know from \cite[Section 4.2]{MR3912678} that
\begin{equation}\label{rhom:ant}
\int _V |\overline{\rho}-\rho_m|^{-\beta+1}  \leq c(R)\mbox{  for all  }t\in [0,T]. \end{equation}

We can use \eqref{m:rhoCw}--\eqref{m:rhos} to conclude 
\begin{equation}\label{bd:rhoR1}
0\leq \rho_R \leq \overline{\rho}.
\end{equation}
Using \eqref{E2} and taking limit $m\rightarrow\infty$, we have
\begin{equation}\label{ER}
\sup_{t\in (0,T)}\int _V E(\rho _R|\rho_{\infty})
\mbox{d}x \leq c(R). \end{equation}
Observe that 
\begin{equation*}
\lim_{\rho_R\rightarrow \overline{\rho}} \frac{E(\rho_R|\rho_{\infty})}{H(\rho_R)} = 1.
\end{equation*}
Thus, there exists $\delta>0$ such that 
\begin{equation}\label{rel:HR}
\frac{1}{2}H(\rho_R)\leq E(\rho_R|\rho_{\infty})\leq \frac{3}{2}H(\rho_R),\quad\forall\ \rho_R\in [\overline{\rho}-\delta,\overline{\rho}].
\end{equation}
Moreover, the behaviour of pressure \eqref{p-law} and definition of the function $H$ in \eqref{def:H} imply that
 \begin{equation}\label{HR:beh}
 H(\rho_R)\sim _{\rho\rightarrow \overline{\rho}-} |\overline{\rho}-\rho_R|^{-\beta+1}\mbox{ for some }\beta > 5/2.
 \end{equation}
The relations \eqref{ER}, \eqref{rel:HR} and \eqref{HR:beh} help us to exclude the possibility of equality $\rho_R=\overline{\rho}$ in \eqref{bd:rhoR1} and we can conclude 
\begin{equation}\label{bd:rhoR2}
0\leq \rho_R < \overline{\rho}.
\end{equation}

 \underline{Step 2: Uniform integrability of the pressure.}
 In order to pass the limit $m\rightarrow\infty$ in the weak formulation of the momentum equation \eqref{eq:weakpenmom}, we need the uniform bound of the pressure with respect to the parameter $m$. This is our aim to achieve in this step. We choose cut-off functions $\eta\in C_c^{\infty}(0,T)$ with $0\leq \eta \leq 1$ and $\psi\in C_c^1(B_R\setminus \mc{S})$ with $0\leq \psi \leq 1$. We consider the following test functions 
 \begin{equation}\label{testm}
 \varphi=\eta(t)\mc{B}\left(\psi\rho_m -\psi \alpha_m\right),\quad\mbox{ where }\quad\alpha_m=\frac{\int_{B_R\setminus \mc{S}} \psi\rho_m}{\int_{B_R\setminus \mc{S}} \psi}.
 \end{equation}
 in the weak formulation of momentum equation \eqref{eq:weakpenmom}, where $\mathcal{B}$ is the Bogovskii operator which assigns to each $g\in L^p(B_R\setminus \mc{S})$, $\int\limits_{B_R\setminus \mc{S}} g\ dx=0$, a solution to the problem
 \begin{equation*}
 \operatorname{div} \mc{B}[g] = g \mbox{ in }B_R\setminus \mc{S},\quad \mc{B}[g] =0 \mbox{  on  }\partial B_R \cup \partial\mc{S}.
 \end{equation*}
 Here $\mc{B}$ is a bounded linear operator from $L^p(B_R\setminus \mc{S})$ to $W^{1,p}_0(B_R\setminus \mc{S})$, for any $1< p < \infty$ and it can be extended as a bounded linear operator on $[W^{1,p}(B_R\setminus \mc{S})]'$ with values in $L^{p'}(B_R\setminus \mc{S})$ for any $1< p < \infty$.

 We test the momentum equation \eqref{eq:weakpenmom} with $\varphi$ defined in \eqref{testm} to obtain the following identity:
 \begin{equation*}
 \int\limits_0^T\int\limits_{V} \eta p(\rho_m) (\psi\rho_m - \psi\alpha_m)\ dx\ dt = \sum\limits_{i=1}^{5} I_i,
 \end{equation*}
 where 
\begin{equation*}
I_1= -\int\limits_0^T \partial_t\eta \int\limits_{V} \rho_m\vu_m\cdot \mc{B}(\psi\rho_m - \psi\alpha_m)\ dx\ dt,
\end{equation*}
 \begin{equation*}
I_2= \int\limits_0^T \eta \int\limits_{V} \rho_m\vu_m\cdot \mc{B}(\operatorname{div}(\rho_m\vu_m\psi))\ dx\ dt,
\end{equation*}
  \begin{equation*}
I_3= -\int\limits_0^T \eta \int\limits_{V} \rho_m\vu_m\cdot  \mc{B}\left(\rho_m\vu_m\cdot \nabla \psi - \frac{\psi}{\int_V \psi \ dx}\int_V \rho_m\vu_m \cdot \nabla\psi \ dx\right)\ dx\ dt,
\end{equation*}
  \begin{equation*}
I_4= -\int\limits_0^T \eta \int\limits_{V} \rho_m\vu_m \otimes \vu_m : \nabla \mc{B}(\psi\rho_m - \psi\alpha_m)\ dx\ dt,
\end{equation*}
   \begin{equation*}
I_5= \int\limits_0^T \eta \int\limits_{V} \mathbb{S}(\nabla\vu_m) : \nabla \mc{B}(\psi\rho_m - \psi\alpha_m)\ dx\ dt.
\end{equation*}
The uniform bounds obtained in \eqref{E1}, \eqref{E2}, and \eqref{E4}--\eqref{bd:rhou} along with boundedness of the operator $\mc{B}$ from $L^p(B_R\setminus \mc{S})$ to $W^{1,p}_0(B_R\setminus \mc{S})$ (see \cite[Chapter 3]{MR1284205}), we have that the integrals $I_i$, $i=1,\ldots , 5$ are uniformly bounded with respect to $m$. Consequently, we have 
 \begin{equation}\label{12:43}
 \left|\int\limits_0^T\int\limits_{V} \eta p(\rho_m) (\psi\rho_m - \psi\alpha_m)\ dx\ dt \right| \leq c ,
 \end{equation}
where $c$ is independent of parameter $m$. We also know from \eqref{ini:m} that 
\begin{equation*}
\frac{1}{|V|}\int\limits_V \rho_m \ dx = \frac{1}{|V|}\int\limits_V \rho_0 \ dx = \mc{M}_{0,V} < \overline{\rho}.
\end{equation*}
We can write 
\begin{equation}\label{decom:J1J2}
\int\limits_0^T\int\limits_{V} \eta p(\rho_m) (\psi\rho_m - \psi\alpha_m)\ dx\ dt = J_1 + J_2,
\end{equation}
where 
\begin{equation*}
J_1= \int\limits_0^T\int\limits_{\{\rho_m< (\mc{M}_{0,V}+\overline{\rho})/2\}} \eta p(\rho_m) (\psi\rho_m - \psi\alpha_m)\ dx\ dt,
\end{equation*}
\begin{equation*}
J_2= \int\limits_0^T\int\limits_{\{\rho_m\geq {\mc{M}_{0,V}+\overline{\rho}}/{2}\}} \eta p(\rho_m) (\psi\rho_m - \psi\alpha_m)\ dx\ dt.
\end{equation*}
Since $\eta$, $\psi$ are bounded, the integral $J_1$ is also bounded and we can bound $J_2$ in the following way:
\begin{equation*}
J_2 \geq \frac{\overline{\rho}-\mc{M}_{0,V}}{2}\int\limits_0^T \int\limits_{\{\rho_m\geq (\mc{M}_{0,V}+\overline{\rho})/2\}}\eta\psi p(\rho_m)\ dx\ dt. 
\end{equation*}
Thus, we conclude from estimate \eqref{12:43}, decomposition \eqref{decom:J1J2} and bounds of $J_1$, $J_2$ that
for any compact set $K\subset B_R\setminus \mc{S}$:
\begin{equation}\label{bdd:pm}
\|p(\rho_m)\|_{L^1(0,T;L^1(K))} \leq c(K).
\end{equation}
Since the pressure satisfies the hypothesis \eqref{p-law}, in particular, for  $\delta>0$, we also have \begin{equation}\label{rhom:beta}
\int\limits_0^T\int\limits_{K \cap \{0\leq \rho\leq \overline{\rho}-\delta\}} |\overline{\rho}-\rho_m|^{-\beta} \leq c(K).
\end{equation}
 \underline{Step 3: Equi-integrability of the pressure.} The $L^1(0,T;L^1(K))$ boundedness of the pressure $\{p(\rho_m)\}$ obtained in \eqref{bdd:pm} is not sufficient to pass the limit in the pressure term. We know from Dunford-Pettis Theorem that we need to establish equi-integrability of the pressure family $\{p(\rho_m)\}$ to obtain the weak convergence of the pressure. As in the previous section, we fix the cut-off functions $\eta\in C_c^{\infty}(0,T)$ with $0\leq \eta \leq 1$ and $\psi\in C_c^1(B_R\setminus \mc{S})$ with $0\leq \psi \leq 1$. We consider the following test functions 
 \begin{equation}\label{test1m}
 \varphi_m=\eta(t)\mc{B}\left(\psi b(\rho_m) - \alpha_m\right),\quad\mbox{ where }\quad\alpha_m=\frac{\int_{B_R\setminus \mc{S}} \psi b(\rho_m)}{|{B_R\setminus \mc{S}}|}
 \end{equation}
 with 
 \begin{align}\label{def:b}
 b(\rho)=
 \begin{cases}
 \log (\overline{\rho}-\rho)&\mbox{ if }\rho \in [0, \overline{\rho}-\delta),\\
 \log \delta &\mbox{ if }\rho \geq \overline{\rho}-\delta,
 \end{cases}
 \end{align}
 for some $\delta >0$. Observe that 
 \begin{equation*}
 b'(\rho)=\frac{1}{\overline{\rho}-\rho}\mathds{1}_{[0, \overline{\rho}-\delta)}(\rho).
 \end{equation*}
 We can use estimate \eqref{rhom:ant}, the boundedness \eqref{bdd:rhom} of $\rho_m$, estimate  \eqref{rhom:beta} and the expressions of $b(\rho)$, $b'(\rho)$ to obtain: for any $1\leq p< \infty$ and any compact set $K\subset B_R\setminus \mc{S}$,
 \begin{align}
 \|b(\rho_m)\|_{L^{\infty}(0,T;L^p(B_R\setminus \mc{S}))} &\leq c(p)\label{b1},\\
 \|\rho_m b'(\rho_m)-b(\rho_m)\|_{L^{\beta}((0,T)\times K)} &\leq c(K)\label{b2},\\
 \|\rho_m b'(\rho_m)-b(\rho_m)\|_{L^{\infty}(0,T; L^{\beta-1}(B_R\setminus \mc{S}))} &\leq c\label{b3}.
 \end{align}
 We consider the momentum equation \eqref{eq:weakpenmom}  with $\varphi_m$ (defined in \eqref{test1m}) as a test function and use renormalized equation \eqref{eq:penrenorm} to obtain the following identity:
 \begin{equation*}
 \int\limits_0^T \eta\int\limits_{B_R\setminus \mc{S}}\psi p(\rho_m) b(\rho_m)\ dx\ dt = \sum\limits_{i=1}^{7} I_i,
 \end{equation*}
 where 
 \begin{equation*}
I_1= \frac{1}{|B_R\setminus \mc{S}|}\int\limits_0^T \eta (t) \int\limits_{B_R\setminus \mc{S}} \psi b(\rho_m)\ dx \int\limits_V p( \rho_m)\ dx\ dt,
\end{equation*}
\begin{equation*}
I_2= \int\limits_0^T \partial_t\eta \int\limits_{V} \rho_m\vu_m\cdot \mc{B}(\psi b(\rho_m) - \alpha_m)\ dx\ dt,
\end{equation*}
 \begin{equation*}
I_3= \int\limits_0^T \eta \int\limits_{V} \rho_m\vu_m\cdot \mc{B}(\operatorname{div}(\rho_m\vu_m\psi))\ dx\ dt,
\end{equation*}
  \begin{equation*}
I_4= -\int\limits_0^T \eta \int\limits_{V} \rho_m\vu_m\cdot  \mc{B}\left(b(\rho_m)\vu_m\cdot \nabla \psi - \frac{1}{|B_R \setminus \mc{S}|}\int_V b(\rho_m)\vu_m \cdot \nabla\psi \ dx\right)\ dx\ dt,
\end{equation*}
\begin{equation*}
I_5= -\int\limits_0^T \eta \int\limits_{V} \rho_m\vu_m\cdot  \mc{B}\left[\psi(\rho_mb'(\rho_m)-b(\rho_m))\operatorname{div}\vu_m -\frac{1}{|B_R \setminus \mc{S}|}\int\limits_{B_R\setminus \mc{S}} \psi(\rho_mb'(\rho_m)-b(\rho_m))\operatorname{div}\vu_m  \ dx\right]\ dx\ dt,
\end{equation*}
  \begin{equation*}
I_6= -\int\limits_0^T \eta \int\limits_{V} \rho_m\vu_m \otimes \vu_m : \nabla \mc{B}(\psi b(\rho_m) - \alpha_m)\ dx\ dt,
\end{equation*}
   \begin{equation*}
I_7= \int\limits_0^T \eta \int\limits_{V} \mathbb{S}(\nabla\vu_m) : \nabla \mc{B}(\psi\rho_m - \alpha_m)\ dx\ dt.
\end{equation*}
 We follow \cite[Section 4.4]{MR3912678}, \cite[Section 3.5]{MR2646821} and use the bounds \eqref{E1}, \eqref{E2}, \eqref{E4}--\eqref{bd:rhou}, \eqref{b1}--\eqref{b3} along with the boundedness of the operator $\mc{B}$ from $[W^{1,q}(B_R\setminus \mc{S})]^{*}$ to $L^{q'}(B_R\setminus \mc{S})$ (see \cite{MR2240056}) to conclude that for any compact set $K\subset B_R\setminus \mc{S}$:
\begin{equation}\label{eqint:pm}
\|p(\rho_m)b(\rho_m)\|_{L^1(0,T;L^1(K))} \leq c(K).
\end{equation}
The estimate \eqref{eqint:pm}
implies that the sequence $p(\rho_m)$ is equi-integrable in $L^1((0,T)\times K)$ and we also have
\begin{equation}\label{prholim}
p(\rho_m) \rightarrow \overline{p(\rho_R)} \quad \mbox{weakly in }L^1((0,T)\times K)
\end{equation}
for any compact set $K\subset B_R\setminus \mc{S}$ up to a subsequence.

 \underline{Step 4: Limit in the momentum equation.}
 We already have from \eqref{con:rhou} that \begin{equation*}
 \rho_m \vu_m \rightarrow \rho_R\vu_R \mbox{ weakly in }L^2(0,T; L^{\frac{6q}{6+q}}(V))  \mbox{  and   weakly}-* \mbox{ in }L^{\infty}(0,T;L^{2}(V)).
  \end{equation*}
  We use the bounds \eqref{bd:rhou}--\eqref{bd:rhouu} and \eqref{prholim} in the momentum equation \eqref{eq:weakpenmom} to obtain the equicontinuity of the sequence $t\mapsto \int\limits_{V} \rho_m\vu_m \phi$, $\phi\in C_c^1(V)$ in $C[0,T]$. Moreover, $\rho_m\vu_m$ is uniformly bounded in $L^2(V)$. We apply Arzela-Ascoli theorem \cite[Lemma 6.2, page 301]{MR2084891} to have
  \begin{equation}\label{m:rhouCw} 
 \rho_m\vu_m\rightarrow \rho_R\vu_R\quad \mbox{ in  }\quad C_{\mbox{weak}}([0,T]; L^2(V)).
 \end{equation}
 Moreover, the compact embedding of $L^2(V)\hookrightarrow\hookrightarrow W^{-1,2}(V)$ ensures that we can apply \cite[Lemma 6.4, page 302]{MR2084891} to have 
 \begin{equation}\label{m:rhoustrl} 
 \rho_m\vu_m\rightarrow \rho_R\vu_R\quad \mbox{ strongly in  }L^{2}(0,T; W^{-1,2}(V)).
 \end{equation}
  The estimates \eqref{bd:rhouu}, \eqref{um} and weak compactness result \cite[Lemma 6.6, page 304]{MR2084891} yield
 \begin{equation}\label{con:rhouum}
 \rho_m \vu_m \otimes \vu_m \rightarrow \rho_R\vu_R \otimes \vu_R \mbox{ weakly in }L^2(0,T; L^{\frac{3}{2}}(V)).
  \end{equation}
 % The weak convergence of $\vu_m$ to $\vu_R$ in $L^2(0,T;W^{1,2}(V))$, compact embedding $W^{1,2}(V)\hookrightarrow L^6(V)$ in combination with the above mentioned weakly$-*$ convergence of $\rho_m \vu_m$ in $L^{\infty}(0,T;L^{2}(V)$ give us
 % \begin{equation}\label{con:rhouu}
 %\rho_m \vu_m \otimes \vu_m \rightarrow \rho_R\vu_R \otimes \vu_R \mbox{ weakly in }L^2(0,T; L^{\frac{3}{2}}(V)).
  %\end{equation}
  Let us recall the notation: $$\Omega_R:= B_R \setminus \overline{\mc{S}}.$$
  We take limit $m \rightarrow \infty$ in 
the momentum equation \eqref{eq:weakpenmom} and use \eqref{rhom}, \eqref{con:rhou}, \eqref{m:rhouCw}--\eqref{con:rhouum} to conclude that: for all $\tau\in [0,T]$ and
 for any test
function $\varphi \in C^1_c([0,T] \times \Omega_R)$,
\begin{multline}\label{limit-m}
\int_{\Omega_R}\rho_R\vu_R(\tau,\cdot)\cdot\varphi(\tau,\cdot){\rm d} x - \int_{\Omega_R} \mathbf{q}_0(\cdot)\cdot\varphi(0,\cdot){\rm d} x  \\
 =\int_0^\tau \int_{\Omega_R}\Big(
\rho_R \vu_R \cdot \partial_t \varphi + \rho_R \vu_R
\otimes \vu_R : \nabla \varphi +
\overline{p(\rho_R)}\operatorname{div}\varphi - \mathbb {S}(\nabla \vu_R) : \nabla \varphi  \Big)\ dx \ dt.
\end{multline}
 Moreover, we know from the estimate \eqref{E3} that
\begin{equation}\label{uRuinf}
\vu_R=\vu_{\infty} \mbox{ a.e. in }(0,T)\times [(V\setminus B_R)\cap \mc{S}].
\end{equation}
  
 We want to show that 
 \begin{equation*}
 \overline{p(\rho_R)}=p(\rho_R).
\end{equation*}
 It is enough to establish that the family of densities $\{\rho_m\}$ converges almost everywhere in $B_R \setminus \mc{S}$. In order to show this, we need to use the idea of \textit{effective viscous flux} which is developed in \cite{MR1637634, MR1867887} in the context of isentropic compressible fluid (see also \cite[Proposition 7.36, Page 338]{MR2084891}). Let us denote by $\nabla \Delta^{-1}$ the pseudodifferential operator of the Fourier symbol $\frac{i\xi}{|\xi|^2}$. We use the test function
 \begin{equation*}
 \varphi(t,x)=\eta (t)\psi (x) \nabla \Delta^{-1} [\rho_m \psi],\quad \eta\in C_c^{\infty}(0,T),\ 0\leq \eta \leq 1,\ \psi\in C_c^1(B_R\setminus \mc{S}),\ 0\leq \psi \leq 1
 \end{equation*}
 in the $m$-th level approximating momentum equation \eqref{eq:weakpenmom} and the test function 
 \begin{equation*}
 \varphi(t,x)=\eta (t)\psi (x) \nabla \Delta^{-1} [\rho_R \psi],\quad \eta\in C_c^{\infty}(0,T),\ 0\leq \eta \leq 1,\ \psi\in C_c^1(B_R\setminus \mc{S}),\ 0\leq \psi \leq 1
 \end{equation*}
 in the limiting momentum equation \eqref{limit-m}. We subtract these two resulting identities and taking the limit $m\rightarrow \infty$. These limiting procedure is the main step in the barotropic Navier-Stokes case (see \cite{MR1637634}, \cite[Lemma 3.2]{MR1867887}). This procedure has been adapted in \cite[Section 3.6]{MR2646821}, \cite[Section 4.6]{MR3912678} in the context of hard pressure law and we follow it here to obtain:  
 \begin{equation}\label{eq:EVF}
 (2\mu+\lambda) \left(\overline{\rho_R\operatorname{div}\vu_R}-\rho_R\operatorname{div}\vu_R \right) =   \left(\overline{p(\rho_R) \rho_R}-\overline{p(\rho_R)}\rho_R\right),
 \end{equation}
 where the quantity $\left(p(\rho)-(2\mu+\lambda)\operatorname{div}\vu\right)$ is termed as \textit{effective viscous flux}. In the above relation and in the sequel the overlined quantities are used to denote the $L^1$-weak limits of the corresponding sequences.
We already established in Step 1 that the bounded density $\rho_R\in C_{\mbox{weak}}([0,T]; L^1(V))$ satisfies the continuity equation
\eqref{eq:weakmmass}.  We can use the regularization procedure in the transport theory by DiPerna and Lions \cite{DiPerna1989} to show that $(\rho_R,\vu_R)$ also satisfy a renormalized continuity equation in a weak sense, i.e,
 \begin{equation}\label{eq:renormR}
\partial_t b(\rho_R) + \operatorname{div}(b(\rho_R)\vu_R) + (b'(\rho_R)-b(\rho_R))\operatorname{div}\vu_R=0 \mbox{ in }\, \mc{D}'([0,T)\times \overline{V}) ,
\end{equation} 
for any $b\in C([0,\infty)) \cap C^1((0,\infty))$. Now use $b(\rho_R)=\rho_R\log\rho_R$ in \eqref{eq:renormR} and subtract the resulting equation to the equation \eqref{eq:penrenorm} with $b(\rho_m)=\rho_m\log\rho_m$ and letting $m\rightarrow \infty$ to get
\begin{equation}\label{renorm:diff}
\int\limits_{V} \left(\overline{\rho_R \log \rho_R}-\rho_R \log \rho_R\right)(\tau,\cdot)\ dx = \int\limits_0^{\tau}\int\limits_{V} \left(\rho_R\operatorname{div}\vu_R - \overline{\rho_R\operatorname{div}\vu_R}\right)\ dx\ dt .
 \end{equation}
 Now we use the relation \eqref{eq:EVF} in the inequality \eqref{renorm:diff} to obtain 
 \begin{equation}\label{13:44}
\int\limits_{V} \left(\overline{\rho_R \log \rho_R}-\rho_R \log \rho_R\right)(\tau,\cdot)\ dx = \frac{1}{2\mu+\lambda} \int\limits_0^{\tau}\int\limits_{V} \left(\overline{p(\rho_R)}\rho_R-\overline{p(\rho_R) \rho_R}\right)\ dx\ dt .
 \end{equation}
 Since the function $\rho \mapsto \rho\log\rho$ is a convex lower semi-continuous function on $(0,T)$ with 
 \begin{equation*}
 \rho_m \rightarrow \rho_R\quad\mbox{ weakly in }\quad L^1(V),\quad
 \rho_m\log\rho_m \rightarrow \overline{\rho_R\log\rho_R}\quad\mbox{ weakly in }\quad L^1(V),
 \end{equation*}
 then we can apply the result \cite[Corollary 3.33, Page 184]{MR2084891} from convex analysis to conclude
 \begin{equation}\label{c1}
 \rho_R\log\rho_R \leq \overline{\rho_R\log\rho_R}\quad\mbox{ a.e. in }V.
 \end{equation}
 Moreover, we know from the assumption \eqref{p-law} on pressure that $p(\rho)$ is strictly increasing on $(0,\infty)$, that is $p'(\rho)>0\ \forall\ \rho>0$. We can use the result  \cite[Lemma 3.35, Page 186]{MR2084891} on weak convergence and monotonicity to get: 
  \begin{equation}\label{c2}
 \overline{p(\rho_R)\rho_R} \geq \overline{p(\rho_R)}\rho_R\quad\mbox{ a.e. in }V.
 \end{equation}
 Thus, the relations \eqref{c1}--\eqref{c2} and the inequality \eqref{13:44} yield
  \begin{equation*}
 \rho_R\log\rho_R = \overline{\rho_R\log\rho_R}\quad\mbox{ a.e. in }V.
 \end{equation*}
 Since the function $\rho \mapsto \rho\log\rho$ is strictly convex on $[0,\infty)$, we obtain that 
  \begin{equation}\label{lim:rho}
 \rho_m \rightarrow {\rho_R}\quad\mbox{ a.e. in }(0,T)\times V \mbox{ and in }L^p((0,T)\times V)\mbox{ for }1\leq p< \infty.
 \end{equation}
 We use \eqref{lim:rho} together with \eqref{prholim} to conclude that for any compact set $K\subset B_R\setminus \mc{S}$:
\begin{equation*}
p(\rho_m) \rightarrow {p(\rho_R)} \quad \mbox{ a.e. in }(0,T)\times V \mbox{ and in }L^1((0,T)\times K).
\end{equation*}
In particular, we have $\overline{p(\rho_R)}=p(\rho_R)$. The substitution of this relation in \eqref{limit-m} yields the limiting momentum equation : for all $\tau\in [0,T]$ and
 for any test
function $\varphi \in C^1_c([0,T] \times \Omega_R)$,
\begin{multline*}
\int_{\Omega_R}\rho_R\vu_R(\tau,\cdot)\cdot\varphi(\tau,\cdot){\rm d} x - \int_{\Omega_R} \mathbf{q}_0(\cdot)\cdot\varphi(0,\cdot){\rm d} x  \\
 =\int_0^\tau \int_{\Omega_R}\Big(
\rho_R \vu_R \cdot \partial_t \varphi + \rho_R \vu_R
\otimes \vu_R : \nabla \varphi +
{p(\rho_R)}\operatorname{div}\varphi - \mathbb {S}(\nabla \vu_R) : \nabla \varphi  \Big)\ dx \ dt.
\end{multline*}
Hence, we have verified the momentum equation \eqref{eq:weakmmom} and it only remains to establish the energy inequality \eqref{m:energy}. 
 
 \underline{Step 5: Energy inequality.} Let us recall the energy inequality \eqref{pen:energy} for the $m$-th level penalized problem: for a.e. $\tau\in (0,T)$,
\begin{multline}\label{pen:energyag}
\int_{V}\Big(\frac{1}{2}\rho_m|\vu_m-\mathbf {u}_{\infty}|^2 +
E(\rho_m|\rho_{\infty})\Big)(\tau)\ dx
+ \int_0^\tau\int_{V}\Big( \mathbb{S}(\nabla (\vu_m-\vu_{\infty})):\nabla (\vu_m-\vu_{\infty})  +m\mathds{1}_{\{(V\setminus B_R)\cup \mc{S}\}} |\vu_m-\vu_{\infty}|^2\Big)\ dx\ dt\\
\leq
  \int_{V}\Big(\frac{1}{2}\frac{|\mathbf{q}_0-\rho_{0}\mathbf{u}_{\infty}|^2}{\rho_0} +
E(\rho_0|\rho_{\infty})\Big)\ dx + \int_0^\tau
\int_{V} \rho_m\vu_m\cdot\nabla\mathbf {u}_{\infty}
\cdot(\mathbf{u}_{\infty}- \vu_m)\,{ d}x\,{ d}t - \int _0^{ \tau}\int _V\mathbb{S}(\nabla\vu_{\infty}):\nabla(\vu_m-\vu_\infty)\,{ d} x\, { d}t.
\end{multline}
We use
\begin{itemize}
\item the properties \eqref{d2} of $\vu_{\infty}$,
\item the relation \eqref{uRuinf}, that is  $\vu_R=\vu_{\infty}$ on the set $[(V\setminus B_R)\cup \mc{S}]$, 
\item the convergences obtained for $\rho_m$, $\vu_m$, $\rho_m\vu_m$, \item the lower semi-continuity of the convex functionals at the left-hand side of \eqref{pen:energyag},
\end{itemize}
and take the limit $m \rightarrow \infty$ in \eqref{pen:energyag} to obtain
\begin{multline*}
\int_{\Omega_R}\Big(\frac{1}{2}\rho_R|\vu_R-\mathbf {u}_{\infty}|^2 +
E(\rho_R|\rho_{\infty})\Big)(\tau)\ dx
+ \int_0^\tau\int_{\Omega_R} \mathbb{S}(\nabla (\vu_R-\vu_{\infty})):\nabla (\vu_R-\vu_{\infty})\ dx\ dt\\
\leq
  \int_{\Omega_R}\Big(\frac{1}{2}\frac{|\mathbf{q}_0-\rho_{0}\mathbf{u}_{\infty}|^2}{\rho_0} +
E(\rho_0|\rho_{\infty})\Big)\ dx + \int_0^\tau
\int_{B_1\setminus\mc{S}} \rho\vu\cdot\nabla\mathbf {u}_{\infty}
\cdot(\mathbf{u}_{\infty}- \vu)\,{ d}x\,{ d}t - \int _0^{ \tau}\int _{B_1\setminus\mc{S}}\mathbb{S}(\nabla\vu_{\infty}):\nabla(\vu-\vu_\infty)\,{ d} x\, { d}t,
\end{multline*}
for a.e. $\tau\in (0,T)$. Thus we have established the existence of at least one renormalized bounded energy weak solution $(\rho_R,\vu_R)$ of the problem \eqref{eq:mmass}--\eqref{mini} according to \cref{def:mbddenergy}.
 \end{proof}
 %%%%%%%%%%%%%%%%%%%%%%%%%%%%%%%%%%%%%%%%%%%%
 
 \section{Proof of the main result}\label{proof}
 
In this section, we want to show the existence of the solution $(\rho,\vu)$ according to \cref{def:bddenergy} to the system \eqref{eq:mass}--\eqref{behaviour}  as a limit of the solution $(\rho_R,\vu_R)$ to the system   \eqref{eq:mmass}--\eqref{mini} as $R\rightarrow \infty$. The existence of such a sequence $\{(\rho_R,\vu_R)\}$ to problem \eqref{eq:mmass}--\eqref{mini} on $\Omega_R=B_R\setminus \overline{\mc{S}}$ is guaranteed by \cref{thm:m}. In order to prove \cref{thm:main}, we are going to follow the idea of \cite[Section 7.11]{MR2084891} concerning the method of `` invading domains'' to obtain existence in unbounded domain.

\begin{proof}[Proof of \cref{thm:main}]
Let us define
\begin{align}\label{def:Uinf}
\mathbf{U}_{\infty} := \begin{cases}
\vu_{\infty}&\mbox{ in }B_1,\\
\mathbf{a}_{\infty}&\mbox{ in }\mathbb{R}^3\setminus B_1,
\end{cases}
\end{align}
where $\vu_{\infty}$ is given by \eqref{d2}.
Now we extend $\rho_R$ by $\rho_{\infty}$ and $\vu_R$ by $\mathbf{U}_{\infty}$ outside $\Omega_{R}$ and with the abuse of notation, we still denote the new functions by $\rho_R$, $\vu_R$ respectively.

 \underline{Step 1: Some convergences.}

We denote by $C$, a generic constant that is independent of $R$ but depends on $\mathbf{U}_{\infty}$ and initial data $(\rho_0,\mathbf{q}_0)$. As $(\rho_R,\vu_R)$ satisfies energy inequality \eqref{m:energy}, we can deduce the following estimates: \begin{equation}\label{est:R1}
\displaystyle 
\sup_{t\in (0,T)}\int _{\mathbb{R}^3} \rho _R|\mathbf u_R-\mathbf{U}_{\infty}|^2
\mbox{d}x \leq C,
\end{equation}
\begin{equation}\label{est:R2}
\sup_{t\in (0,T)}\int _{\mathbb{R}^3} E(\rho _R|\rho_{\infty})
\mbox{d}x \leq C, \end{equation}
\begin{equation}\label{est:R3}
\|\vu_R - \mathbf{U}_{\infty} \|_{L^2(0,T;W^{1,2}(\mathbb{R}^3))}\le C.
\end{equation}
 Recall that 
 \begin{equation*}
E(\rho_R|\rho_{\infty}) =
H(\rho_R)-H'(\rho_{\infty})(\rho_R-\rho_{\infty})-H(\rho_{\infty}),
\end{equation*}
where
\begin{equation*}
H(\rho) = \rho \int _1^{\rho} \frac{p(s)}{s^2}ds.
\end{equation*}
Then we have 
 \begin{equation}\label{est:ER}
E(\rho_R|\rho_{\infty}) =H''(\xi)|\rho_R-\rho_{\infty}|^2=\frac{p'(\xi)}{\xi}|\rho_R-\rho_{\infty}|^2, \mbox{ for some }\xi\in (0,\overline{\rho}).
\end{equation}
Now, the assumption \eqref{p-law} on pressure \begin{equation*}
 p'(\rho)>0\ \forall\ \rho>0,\quad \liminf_{\rho\rightarrow 0}\frac{p'(\rho)}{\rho}>0,
\end{equation*} 
along with the relations \eqref{est:R2},
\eqref{est:ER} yield
\begin{equation}\label{beh:rhoR}
 \sup_{t\in (0,T)}\int_{\mathbb{R}^3} |\rho_R-\rho_{\infty}|^2\mathds{1}_{\{0\leq \rho_R< \overline{\rho}\}} \leq C. \end{equation}
 Thus, we obtain from \eqref{est:R3} and \eqref{beh:rhoR} that
 \begin{equation}\label{diff:rhoR}
 \rho_R-\rho_{\infty}
 \rightarrow \rho -\rho_{\infty}\quad\mbox{weakly in }L^{\infty}(0,T;L^2(\mathbb{R}^3)),
 \end{equation}
 \begin{equation}\label{diff:uR}
  \vu_R-\mathbf{U}_{\infty}\rightarrow \vu -\mathbf{U}_{\infty}\quad\mbox{weakly in }L^{2}(0,T;W^{1,2}(\mathbb{R}^3)) . 
  \end{equation}
 Let us fix $n\in \mathbb{N}$ and define:
 \begin{equation*}
 \Omega_n:= B_n \setminus \overline{\mc{S}}.
 \end{equation*}
 We consider the test function $\phi_R$ (see \eqref{test1m}--\eqref{def:b}) and use it in the momentum equation \eqref{eq:weakmmom}. We use renormalized equation \eqref{eq:renormR} and follow the same procedure as in step 3 of the proof of \cref{thm:m} to have equi-integrability of the sequence of integrable functions $\{p(\rho_R)\}$ and conclude that
 \begin{equation}\label{fl:pressure}
 p(\rho_R)\rightarrow \overline{\rho(\rho)}\quad\mbox{weakly in }\quad L^1((0,T)\times \Omega_n).
 \end{equation}

 \underline{Step 2: Limit in the continuity equation.}

 We use the estimates \eqref{est:R1}, \eqref{est:R3} and boundedness of $\rho_R$ in $(0,T)\times \mathbb{R}^3$, we have
 \begin{equation}\label{bd:rhouR}
 \|\rho_R \vu_R\|_{L^{\infty}(0,T;L^{2}(\Omega_n))} + \|\rho_R \vu_R\|_{L^2(0,T; L^{\frac{6q}{6+q}}(\Omega_n))} \leq C,\mbox{  for any  }1\leq q <\infty,
 \end{equation}
  \begin{equation}\label{bd:rhouuR}
 \|\rho_R |\vu_R|^2\|_{L^2(0,T; L^{\frac{3}{2}}(\Omega_n))} \leq C.
 \end{equation}
  Furthermore, for any fixed $n\in \mathbb{N}$, we deduce from the continuity equation \eqref{eq:weakmmass} and the estimate \eqref{bd:rhouR} that the sequence of functions $t\mapsto \int\limits_{\Omega_n} \rho_R \varphi$, $\varphi \in C_c^1(\Omega_n)$ is equi-continuous. The density $\rho_R$ is uniformly bounded.  We can apply Arzela-Ascoli theorem \cite[Lemma 6.2, page 301]{MR2084891} to have $\rho_R\rightarrow \rho|_{B_n}$ in $C_{\mbox{weak}}([0,T]; L^q(\Omega_n))\mbox{  for any  }1\leq q <\infty$. Using the diagonalization process, we thus obtain 
  \begin{equation}\label{R:rhoCw} 
 \rho_R\rightarrow \rho\quad \mbox{ in  }\quad C_{\mbox{weak}}([0,T]; L^q(\Omega_n)),\ n\in \mathbb{N},\mbox{  for any  }1\leq q <\infty.
 \end{equation}
 Moreover, the compact embedding of $L^q(\Omega_n)\hookrightarrow\hookrightarrow W^{-1,2}(\Omega_n)$, for $q>6/5$ ensures that we can apply \cite[Lemma 6.4, page 302]{MR2084891} to have 
 \begin{equation}\label{R:rhostr} 
 \rho_R\rightarrow \rho\quad \mbox{ strongly in  }L^{2}(0,T; W^{-1,2}(\Omega_n)),\ n\in \mathbb{N}.
 \end{equation} 
 We combine the convergence \eqref{R:rhoCw}--\eqref{R:rhostr} of density $\rho_R$, convergence \eqref{diff:uR} of $\vu_R$, boundedness \eqref{bd:rhouR} of $\rho_R\vu_R$ and local weak compactness result \cite[Lemma 6.6, page 304]{MR2084891} in unbounded domains to obtain 
 \begin{equation}\label{conv:rhouR}
 \rho_R \vu_R \rightarrow \rho\vu \mbox{ weakly-}*\mbox{ in }{L^{\infty}(0,T;L^{2}(\Omega_n))}\mbox{ and weakly in }  {L^2(0,T; L^{\frac{6q}{6+q}}(\Omega_n))} ,\mbox{  for any  }1\leq q <\infty,\ n\in \mathbb{N}.
 \end{equation}
 Consequently, the convergences \eqref{R:rhoCw}--\eqref{conv:rhouR} enable us to pass the limit $R\rightarrow \infty$ in the equation \eqref{eq:weakmmass} and we obtain:
\begin{equation*}
 \int_\Omega \rho(\tau,\cdot)\varphi(\tau,\cdot)\ {\rm
d}x - \int_\Omega \rho_0(\cdot)\varphi(0,\cdot)\ {\rm
d}x=\int_0^\tau \int_{\Omega}\left(\rho
\partial_t \varphi + \rho\vu \cdot \nabla \varphi\right)
 \ dx\ dt,
\end{equation*}
for all $\tau\in [0,T]$ and any test function $\varphi
\in C^1_{c}([0,T] \times \overline{\Omega})$.

We can establish as in Step 1 of the proof of \cref{thm:m} that the density $\rho$ is bounded and we have already proved that it  satisfies the continuity equation
\eqref{eqf:weakmass}.  We can use the regularization procedure in the transport theory by DiPerna and Lions \cite{DiPerna1989} to show that $(\rho,\vu)$ also satisfy a renormalized continuity equation in a weak sense, i.e,
 \begin{equation}\label{eqf:renorm}
\partial_t b(\rho) + \operatorname{div}(b(\rho)\vu) + (b'(\rho)-b(\rho))\operatorname{div}\vu=0 \mbox{ in }\, \mc{D}'([0,T)\times \overline{\Omega}) ,
\end{equation} 
for any $b\in C([0,\infty)) \cap C^1((0,\infty))$.

 \underline{Step 3: Limit in the momentum equation.}
 
 Fix any $n\in \mathbb{N}$, we use the bounds \eqref{bd:rhouR}--\eqref{bd:rhouuR} and \eqref{fl:pressure} in the momentum equation \eqref{eq:weakmmom} to obtain the equicontinuity of the sequence $t\mapsto \int\limits_{\Omega_n} \rho_R\vu_R \phi$, $\phi\in C_c^1(\Omega_n)$ in $C[0,T]$. Moreover, $\rho_R\vu_R$ is uniformly bounded in $L^2(\Omega_n)$. Again, we apply Arzela-Ascoli theorem \cite[Lemma 6.2, page 301]{MR2084891} to have $\rho_R\vu_R\rightarrow \rho\vu|_{B_n}$ in $C_{\mbox{weak}}([0,T]; L^2(\Omega_n))$. Using the diagonalization process, we arrive at
  \begin{equation}\label{R:rhouCw} 
 \rho_R\vu_R\rightarrow \rho\vu\quad \mbox{ in  }\quad C_{\mbox{weak}}([0,T]; L^2(\Omega_n)),\ n\in \mathbb{N}.
 \end{equation}
 Moreover, the compact embedding of $L^2(\Omega_n)\hookrightarrow\hookrightarrow W^{-1,2}(\Omega_n)$ ensures that we can apply \cite[Lemma 6.4, page 302]{MR2084891} to have 
 \begin{equation}\label{R:rhoustr} 
 \rho_R\vu_R\rightarrow \rho\vu\quad \mbox{ strongly in  }L^{2}(0,T; W^{-1,2}(\Omega_n)),\ n\in \mathbb{N}.
 \end{equation}
  The estimates \eqref{bd:rhouuR}, \eqref{diff:uR} and weak compactness result \cite[Lemma 6.6, page 304]{MR2084891} yield
 \begin{equation}\label{con:rhouuR}
 \rho_R \vu_R \otimes \vu_R \rightarrow \rho\vu \otimes \vu \mbox{ weakly in }L^2(0,T; L^{\frac{3}{2}}(B_n)), \ n\in \mathbb{N}.
  \end{equation}
  We take the limit $R \rightarrow \infty$ in 
the momentum equation \eqref{eq:weakmmom} and use convergences \eqref{R:rhoCw}--\eqref{R:rhostr} of $\rho_R$, convergences \eqref{R:rhouCw}--\eqref{R:rhoustr} , convergence \eqref{con:rhouuR} of $\rho_R \vu_R \otimes \vu_R$ and convergence \eqref{fl:pressure} of pressure to conclude that: \begin{multline}\label{limitmom-R}
\int_\Omega\rho\vu(\tau,\cdot)\cdot\varphi(\tau,\cdot){\rm d} x - \int_\Omega \mathbf{q}_0(\cdot)\cdot\varphi(0,\cdot){\rm d} x  \\
 =\int_0^\tau \int_{\Omega}\Big(
\rho \vu \cdot \partial_t \varphi + \rho \vu
\otimes \vu : \nabla \varphi +
\overline{p(\rho)}\operatorname{div}\varphi - \mathbb {S}(\nabla \vu) : \nabla \varphi \Big)\ dx \ dt,
\end{multline}
for all $\tau\in [0,T]$ and
 for any test
function $\varphi \in C^1_c([0,T] \times \Omega)$.

 \underline{Step 4: Identify the limit of the pressure.}

In this step, we want to show that 
 \begin{equation*}
 \overline{p(\rho)}=p(\rho).
\end{equation*}
In order to show this, we need to use the idea of \textit{effective viscous flux} as explained in Step 4 of the proof of \cref{thm:m}. We can establish the identity
 \begin{equation}\label{eqf:EVF}
 (2\mu+\lambda) \left(\overline{\rho\operatorname{div}\vu}-\rho\operatorname{div}\vu \right) =   \left(\overline{p(\rho) \rho}-\overline{p(\rho)}\rho\right),
 \end{equation}
 where the quantity $\left(p(\rho)-(2\mu+\lambda)\operatorname{div}\vu\right)$ is termed as \textit{effective viscous flux}. Now we consider $b(\rho_R)=\rho_R\log\rho_R$ in \eqref{eq:renormR}  letting $R\rightarrow \infty$ to get
 \begin{align}
 \int_{\Omega} \overline{\rho \log \rho}(\tau,\cdot)\varphi(\tau,\cdot)\ {\rm
d}x - \int_{\Omega} \rho_0\log\rho_0(\cdot)\varphi(0,\cdot)\ {\rm
d}x\notag\\
=\int_0^\tau \int_{\Omega}\left(\overline{\rho \log \rho}\
\partial_t \varphi + \overline{\rho \log \rho}\ \vu \cdot \nabla \varphi - \overline{\rho\operatorname{div}\vu}\cdot \varphi\right)
 \ dx\ dt,\label{renorm:Rinf}
\end{align}
for all $\tau\in [0,T]$ and any test function $\varphi
\in C^1_{c}([0,T] \times \overline{\Omega})$. We take $b(\rho)=\rho\log\rho$ in the equation \eqref{eqf:renorm} and subtract it from \eqref{renorm:Rinf} to obtain
\begin{multline}\label{renormf:diff}
 \int_{\Omega} \left(\overline{\rho \log \rho}-{\rho \log \rho}\right)(\tau,\cdot)\varphi(\tau,\cdot)\ {\rm
d}x \\
=\int_0^\tau \int_{\Omega}\left[(\overline{\rho \log \rho}-{\rho \log \rho})\
\partial_t \varphi + (\overline{\rho \log \rho}-{\rho \log \rho})\ \vu \cdot \nabla \varphi - (\overline{\rho\operatorname{div}\vu}-{\rho\operatorname{div}\vu})\cdot \varphi\right]
 \ dx\ dt,
\end{multline}
for all $\tau\in [0,T]$ and any test function $\varphi
\in C^1_{c}([0,T] \times \overline{\Omega})$. Without loss of generality, we suppose that $\mathbf{a}_{\infty}$, the prescribed  behaviour of $\mathbf{u}$ at infinity (see \eqref{behaviour}) is of the following form:
\begin{equation*}
\mathbf{a}_{\infty}=(a_{\infty},0,0),\quad a_{\infty}>0.
\end{equation*}
We define
\begin{equation}\label{def:PhiR1}
\Phi_R (x):= \eta\left(\frac{x_1}{R}\right)\zeta\left(\frac{x'}{R^{\alpha}}\right),\quad x'=(x_2,x_3),\quad R>1,\quad \alpha>0,
\end{equation}
where 
\begin{equation}\label{def:PhiR2}
\eta \in C_c^{\infty}(\mathbb{R}),\quad 0\leq \eta \leq 1,\quad \eta(s)=\begin{cases}
1\mbox{ if }|s|\leq 1,\\ 0 \mbox{ if }|s|\geq 2,
\end{cases}
\quad \zeta \in C_c^{\infty}(\mathbb{R}^2),\quad 0\leq \zeta \leq 1,\quad \zeta(s)=\begin{cases}
1\mbox{ if }|x'|\leq 1,\\ 0 \mbox{ if }|x'|\geq 2.
\end{cases}
\end{equation}
We take $\Phi_R$ as the test function in \eqref{renormf:diff} and use the \textit{effective viscous flux} identity
\eqref{eqf:EVF} to get \begin{equation}\label{renormf:diffR}
 \int_{\Omega} \left(\overline{\rho \log \rho}-{\rho \log \rho}\right)(\tau,\cdot)\Phi_R(\cdot)\ {
d}x + \int_0^\tau \int_{\Omega}(\overline{p(\rho) \rho}-\overline{p(\rho)}\rho)\cdot \Phi_R\ dx \ dt
=\int_0^\tau \int_{\Omega} (\overline{\rho \log \rho}-{\rho \log \rho})\ \vu \cdot \nabla \Phi_R 
 \ dx\ dt.
\end{equation}
Since $p(\rho)$ is strictly increasing on $(0,\infty)$, we can use the result  \cite[Lemma 3.35, Page 186]{MR2084891} on weak convergence and monotonicity to get: 
  \begin{equation*}
 \overline{p(\rho)\rho} \geq \overline{p(\rho)}\rho\quad\mbox{ a.e. in }\Omega.
 \end{equation*}
The above relation and \eqref{renormf:diffR} yield
\begin{equation}\label{rhofinal1}
 \int_{\Omega} \left(\overline{\rho \log \rho}-{\rho \log \rho}\right)(\tau,\cdot)\Phi_R(\cdot)\ {
d}x
\leq \int_0^\tau \int_{\Omega} (\overline{\rho \log \rho}-{\rho \log \rho})\ \vu \cdot \nabla \Phi_R 
 \ dx\ dt.
\end{equation}
Since the function $\rho \mapsto \rho\log\rho$ is a convex lower semi-continuous function on $(0,T)$ with 
 \begin{equation*}
 \rho_R\log\rho_R \rightarrow \overline{\rho\log\rho}\quad\mbox{ weakly in }\quad L^1(\Omega),
 \end{equation*}
 we can apply the result \cite[Corollary 3.33, Page 184]{MR2084891} from convex analysis to conclude
 \begin{equation}\label{11:51}
 \rho\log\rho \leq \overline{\rho\log\rho}\quad\mbox{ a.e. in }\Omega.
 \end{equation}
Furthermore, we estimate the right-hand side of \eqref{rhofinal1} in the following way
\begin{multline}\label{rediff1}
 \int_0^\tau \int_{\Omega} (\overline{\rho \log \rho}-{\rho \log \rho})\ \vu \cdot \nabla \Phi_R 
 \ dx\ dt\\ \leq \int_0^\tau \int_{\Omega} |\overline{\rho \log \rho}-{\rho_{\infty} \log \rho_{\infty}}|\ \vu \cdot \nabla \Phi_R 
 \ dx\ dt + \int_0^\tau \int_{\Omega} |{\rho_{\infty} \log \rho_{\infty}}-{\rho \log \rho}|\ \vu \cdot \nabla \Phi_R 
 \ dx\ dt\\ \leq
 \int_0^\tau \int_{\Omega} |\overline{\rho \log \rho}-{\rho_{\infty} \log \rho_{\infty}}|\ (\vu-\mathbf{U}_{\infty}) \cdot \nabla \Phi_R 
 \ dx\ dt + \int_0^\tau \int_{\Omega} |\overline{\rho \log \rho}-{\rho_{\infty} \log \rho_{\infty}}|\ \mathbf{U}_{\infty} \cdot \nabla \Phi_R 
 \ dx\ dt \\
 +\int_0^\tau \int_{\Omega} |{\rho_{\infty} \log \rho_{\infty}}-{\rho \log \rho}|\ (\vu-\mathbf{U}_{\infty}) \cdot \nabla \Phi_R 
 \ dx\ dt + \int_0^\tau \int_{\Omega} |{\rho_{\infty} \log \rho_{\infty}}-{\rho \log \rho}|\ \mathbf{U}_{\infty} \cdot \nabla \Phi_R 
 \ dx\ dt.
\end{multline}
Observe that the construction of test function $\Phi_R$ in \eqref{def:PhiR1} yield
\begin{equation}\label{prop:PhiR1}
|\mbox{supp}\nabla \Phi_R|\leq CR^{1+2\alpha},\quad \|\nabla \Phi_R\|_{L^p(\mathbb{R}^3)}\leq CR^{\frac{1+2\alpha-\alpha p}{p}},
\end{equation}
\begin{equation}\label{prop:PhiR2}
|\partial_{x_1} \Phi_R|\leq \frac{C}{R},\quad \|\nabla_{x'} \Phi_R\|\leq \frac{C}{R^{\alpha}}.
\end{equation}
Moreover, 
\begin{equation*}
\lim_{s\rightarrow \rho_{\infty}} \frac{s\log s - \rho_{\infty}\log \rho_{\infty}}{s-\rho_{\infty}}= 1 +\log\rho_{\infty},\quad \lim_{s\rightarrow {\infty}} \frac{s\log s - \rho_{\infty}\log \rho_{\infty}}{|s-\rho_{\infty}|^q}=0 \mbox{ for any }1< q < \infty,
\end{equation*}
implies 
\begin{align}
|{\rho_{\infty} \log \rho_{\infty}}-{\rho \log \rho}|&= |{\rho_{\infty} \log \rho_{\infty}}-{\rho \log \rho}|\mathds{1}_{\{|\rho-\rho_{\infty}|\leq 1\}} + |{\rho_{\infty} \log \rho_{\infty}}-{\rho \log \rho}|\mathds{1}_{\{|\rho-\rho_{\infty}|\geq 1\}}\notag\\
&\leq 
C\left(|\rho-\rho_{\infty}|\mathds{1}_{\{|\rho-\rho_{\infty}|\leq 1\}} + |\rho-\rho_{\infty}|^q\mathds{1}_{\{|\rho-\rho_{\infty}|\geq 1\}}\right)\label{13:16}
\end{align}
Then, we can estimate the terms in the right-hand side of \eqref{rediff1} by using H\"{o}lder inequality, properties \eqref{prop:PhiR1}--\eqref{prop:PhiR2} of test function $\Phi_R$ and the estimates \eqref{13:16}, \eqref{est:R3} and \eqref{beh:rhoR} to obtain
\begin{multline}\label{RHS1}
 \int_{\Omega} |{\rho_{\infty} \log \rho_{\infty}}-{\rho \log \rho}|\ (\vu-\mathbf{U}_{\infty}) \cdot \nabla \Phi_R 
 \ dx  \\
 \leq  \|\nabla \Phi_R\|_{L^6(\Omega)}\|\vu-\mathbf{U}_{\infty}\|_{L^6(\Omega)}\|{\rho_{\infty} \log \rho_{\infty}}-{\rho \log \rho}\|_{L^{3/2}(\Omega)}\mathds{1}_{\{|\rho-\rho_{\infty}|\geq 1\}} \\+ \|\nabla \Phi_R\|_{L^{\infty}(\Omega)}\|\vu-\mathbf{U}_{\infty}\|_{L^2(\Omega)}\|{\rho_{\infty} \log \rho_{\infty}}-{\rho \log \rho}\|_{L^{2}(\Omega)}\mathds{1}_{\{|\rho-\rho_{\infty}|\leq 1\}} \leq C\left(R^{\frac{1-4\alpha}{6}} + \frac{1}{R^{\alpha}}\right),
\end{multline}
 \begin{multline}\label{RHS2}
 \int_{\Omega} |{\rho_{\infty} \log \rho_{\infty}}-{\rho \log \rho}|\ \mathbf{U}_{\infty} \cdot \nabla \Phi_R 
 \ dx  \\
 \leq  \|\partial_{x_1} \Phi_R\|_{L^{\infty}(\Omega)}\|\mathbf{U}_{\infty}\|_{L^{\infty}(\Omega)}\|{\rho_{\infty} \log \rho_{\infty}}-{\rho \log \rho}\|_{L^{1}(\Omega)}\mathds{1}_{\{|\rho-\rho_{\infty}|\geq 1\}} \\+ \|\partial_{x_1} \Phi_R\|_{L^{\infty}(\Omega)}\|\mathbf{U}_{\infty}\|_{L^{\infty}(\Omega)}\|{\rho_{\infty} \log \rho_{\infty}}-{\rho \log \rho}\|_{L^{2}(\Omega)}\mathds{1}_{\{|\rho-\rho_{\infty}|\leq 1\}}|\mbox{supp}\nabla \Phi_R|^{1/2} \leq CR^{\frac{2\alpha-1}{2}}.
\end{multline}
Thus, if we choose $\alpha\in \left(\frac{1}{4},\frac{1}{2}\right)$ and take $R\rightarrow \infty$, both the terms in \eqref{RHS1}--\eqref{RHS2} converge to zero. We can treat the other two terms in the right-hand side of \eqref{rediff1} in the same way and combining with the inequality \eqref{11:51}, we conclude that 
\begin{equation*}
 \rho\log\rho = \overline{\rho\log\rho}\quad\mbox{ a.e. in }\Omega.
 \end{equation*}
 Since the function $\rho \mapsto \rho\log\rho$ is strictly convex on $[0,\infty)$, we obtain that 
  \begin{equation}\label{lim:rhoag}
 \rho_R \rightarrow {\rho}\quad\mbox{ a.e. in }(0,T)\times \Omega.
 \end{equation}
In particular, we have $\overline{p(\rho)}=p(\rho)$. The substitution of this relation in \eqref{limitmom-R} yields the limiting momentum equation : for all $\tau\in [0,T]$ and
 for any test
function $\varphi \in C^1_c([0,T] \times \Omega)$,
\begin{multline*}
\int_\Omega\rho\vu(\tau,\cdot)\cdot\varphi(\tau,\cdot){\rm d} x - \int_\Omega \mathbf{q}_0(\cdot)\cdot\varphi(0,\cdot){\rm d} x  \\
 =\int_0^\tau \int_{\Omega}\Big(
\rho \vu \cdot \partial_t \varphi + \rho \vu
\otimes \vu : \nabla \varphi +
p(\rho)\operatorname{div}\varphi - \mathbb {S}(\nabla \vu) : \nabla \varphi \Big)\ dx \ dt,
\end{multline*}
Hence, we have verified the momentum equation \eqref{eqf:weakmom} and it only remains to establish the energy inequality \eqref{eqf:ee}.

 \underline{Step 5: Energy inequality.} Let us recall the energy inequality \eqref{m:energy} for the $R$-th level approximate problem: for a.e. $\tau\in (0,T)$,
\begin{multline*}
\int_{\Omega_R}\Big(\frac{1}{2}\rho_R|\vu_R-\mathbf {u}_{\infty}|^2 +
E(\rho_R|\rho_{\infty})\Big)(\tau)\ dx
+ \int_0^\tau\int_{\Omega_R} \mathbb{S}(\nabla (\vu_R-\vu_{\infty})):\nabla (\vu_R-\vu_{\infty})\ dx\ dt\\
\leq
  \int_{\Omega_R}\Big(\frac{1}{2}\frac{|\mathbf{q}_0-\rho_{0}\mathbf{u}_{\infty}|^2}{\rho_0} +
E(\rho_0|\rho_{\infty})\Big)\ dx + \int_0^\tau
\int_{B_1\setminus\mc{S}} \rho\vu\cdot\nabla\mathbf {u}_{\infty}
\cdot(\mathbf{u}_{\infty}- \vu)\,{ d}x\,{ d}t - \int _0^{ \tau}\int _{B_1\setminus\mc{S}}\mathbb{S}(\nabla\vu_{\infty}):\nabla(\vu-\vu_\infty)\,{ d} x\, { d}t.
\end{multline*}
We use
\begin{itemize}
\item the definition \eqref{def:Uinf} of   $\mathbf{U}_{\infty}$, 
\item the convergences obtained for $\rho_R$, $\vu_R$, $\rho_R\vu_R$, \item the lower semi-continuity of the convex functionals at the left-hand side of \eqref{m:energy},
\end{itemize}
and take the limit $R \rightarrow \infty$ in \eqref{m:energy} to obtain
\begin{multline*}
\int_{\Omega}\Big(\frac{1}{2}\rho|\vu-\mathbf {U}_{\infty}|^2 +
E(\rho|\rho_{\infty})\Big)(\tau)\ dx
+ \int_0^\tau\int_{\Omega} \mathbb{S}(\nabla (\vu-\mathbf{U}_{\infty})):\nabla (\vu-\mathbf{U}_{\infty})\ dx\ dt\\
\leq
  \int_{\Omega}\Big(\frac{1}{2}\frac{|\mathbf{q}_0-\rho_{0}\mathbf{U}_{\infty}|^2}{\rho_0} +
E(\rho_0|\rho_{\infty})\Big)\ dx - \int_0^\tau
\int_{B_1\setminus\mc{S}} \rho\vu\cdot\nabla\mathbf {U}_{\infty}
\cdot(\vu - \mathbf{U}_{\infty})\,{ d}x\,{ d}t - \int _0^{ \tau}\int _{B_1\setminus\mc{S}}\mathbb{S}(\nabla\mathbf{U}_{\infty}):\nabla(\vu-\mathbf{U}_\infty)\,{ d} x\, { d}t.
\end{multline*}
for a.e. $\tau\in (0,T)$. Thus we have established the energy inequality \eqref{eqf:ee} and hence the existence of at least one renormalized bounded energy weak solution $(\rho,\vu)$ of the problem \eqref{eq:mass}--\eqref{p-law} according to \cref{def:bddenergy}.
\end{proof} 

 \section*{Acknowledgment}
{\it \v S. N. and A. R. have been supported by the Czech Science Foundation (GA\v CR) project GA19-04243S. The Institute of Mathematics, CAS is supported by RVO:67985840. The work of A.N.  was partially supported by the distinguished Eduard \v Cech visiting program at the
Institute of Mathematics of the Academy of Sciences of  the Czech Republic.}
 \bibliography{reference}

\begin{thebibliography}{10}

\bibitem{MR2366138}
{\sc F.~Berthelin, P.~Degond, M.~Delitala, and M.~Rascle}, {\em A model for the
  formation and evolution of traffic jams}, Arch. Ration. Mech. Anal., 187
  (2008), pp.~185--220.

\bibitem{MR2438216}
{\sc F.~Berthelin, P.~Degond, V.~Le~Blanc, S.~Moutari, M.~Rascle, and
  J.~Royer}, {\em A traffic-flow model with constraints for the modeling of
  traffic jams}, Math. Models Methods Appl. Sci., 18 (2008), pp.~1269--1298.

\bibitem{MR3974475}
{\sc D.~Bresch, {\v S}.~Ne{\v c}asov{\' a}, and C.~Perrin}, {\em Compression
  effects in heterogeneous media}, J. {\'E}c. Polytech. Math., 6 (2019),
  pp.~433--467.

\bibitem{bresch2014}
{\sc D.~Bresch, C.~Perrin, and E.~Zatorska}, {\em Singular limit of a
  {N}avier-{S}tokes system leading to a free/congested zones two-phase model},
  Comptes Rendus Mathematique, 352 (2014), pp.~685--690.

\bibitem{bresch2017}
{\sc D.~Bresch and M.~Renardy}, {\em Development of congestion in compressible
  flow with singular pressure}, Asymptotic Analysis, 103 (2017), pp.~95--101.

\bibitem{Carnahan-Starling}
{\sc N.~Carnahan and K.~Starling}, {\em Equation of state for nonattracting
  rigid spheres}, J. Chem. Phys., 51 (1969), pp.~635--636.

\bibitem{MR3912678}
{\sc H.~J. Choe, A.~Novotn\'{y}, and M.~Yang}, {\em Compressible
  {N}avier-{S}tokes system with hard sphere pressure law and general
  inflow-outflow boundary conditions}, J. Differential Equations, 266 (2019),
  pp.~3066--3099.

\bibitem{MR3020033}
{\sc J.~Degond, Pierre;~Hua}, {\em Self-organized hydrodynamics with congestion
  and path formation in crowds}, J. Comput. Phys., 237 (2013), pp.~299--319.

\bibitem{MR2835410}
{\sc P.~Degond, J.~Hua, and L.~Navoret}, {\em Numerical simulations of the
  {E}uler system with congestion constraint}, J. Comput. Phys., 230 (2011),
  pp.~8057--8088.

\bibitem{DiPerna1989}
{\sc R.~DiPerna and P.~Lions}, {\em Ordinary differential equations, transport
  theory and {S}obolev spaces.}, Inventiones mathematicae, 98 (1989),
  pp.~511--548.

\bibitem{FLN}
{\sc E.~Feireisl, Y.~Lu, and A.~Novotn{\' y}}, {\em Weak-strong uniqueness for
  the compressible {N}avier-{S}tokes equations with a hard-sphere pressure
  law}, Sci. China Math., 61 (2018), pp.~2003--2016.

\bibitem{MR1867887}
{\sc E.~Feireisl, A.~Novotn\'{y}, and H.~Petzeltov\'{a}}, {\em On the existence
  of globally defined weak solutions to the {N}avier-{S}tokes equations}, J.
  Math. Fluid Mech., 3 (2001), pp.~358--392.

\bibitem{MR2646821}
{\sc E.~Feireisl and P.~Zhang}, {\em Quasi-neutral limit for a model of viscous
  plasma}, Arch. Ration. Mech. Anal., 197 (2010), pp.~271--295.

\bibitem{MR1284205}
{\sc G.~P. Galdi}, {\em An introduction to the mathematical theory of the
  {N}avier-{S}tokes equations. {V}ol. {I}}, vol.~38 of Springer Tracts in
  Natural Philosophy, Springer-Verlag, New York, 1994.
\newblock Linearized steady problems.

\bibitem{MR2808162}
{\sc G.~P. Galdi}, {\em An introduction to the mathematical theory of the
  {N}avier-{S}tokes equations}, Springer Monographs in Mathematics, Springer,
  New York, second~ed., 2011.
\newblock Steady-state problems.

\bibitem{MR2240056}
{\sc M.~Geissert, H.~Heck, and M.~Hieber}, {\em On the equation {${\rm
  div}\,u=g$} and {B}ogovski\u{\i}'s operator in {S}obolev spaces of negative
  order}, in Partial differential equations and functional analysis, vol.~168
  of Oper. Theory Adv. Appl., Birkh\"{a}user, Basel, 2006, pp.~113--121.

\bibitem{KVB62CPES}
{\sc A.~Kastler, R.~Vichnievsky, and G.~Bruhat}, {\em Cours de physique
  g\'en\'erale \`a l'usage de l'enseignement sup\'erieur scientifique et
  technique: Thermodynamique}, Masson et Cie, 1962.

\bibitem{KLM04AEHS}
{\sc J.~Kolafa, S.~Labik, and A.~Malijevsky}, {\em Accurate equation of state
  of the hard sphere fluid in stable and mestable regions}, Phys. Chem. Chem.
  Phys., {\bf 6} (2004), pp.~2335--2340.

\bibitem{MR3208793}
{\sc S.~Kra{\v c}mar, {\v S}.~Ne{\v c}asov{\' a}, and A.~Novotn{\'y}}, {\em The
  motion of a compressible viscous fluid around rotating body}, Ann. Univ.
  Ferrara Sez., Sci. Mat., 60 (2014), pp.~18--208.

\bibitem{MR1637634}
{\sc P.-L. Lions}, {\em Mathematical topics in fluid mechanics. {V}ol. 2},
  vol.~10 of Oxford Lecture Series in Mathematics and its Applications, The
  Clarendon Press, Oxford University Press, New York, 1998.
\newblock Compressible models, Oxford Science Publications.

\bibitem{Hongqin}
{\sc H.~Liu}, {\em {C}arnahan-{S}tarling type equations of state for stable
  hard disk and hard sphere fluids}, Molecular Physics, 119 (2021).

\bibitem{maury2012}
{\sc B.~Maury}, {\em Prise en compte de la congestion dans les modeles de
  mouvements de foules}, Actes des colloques Caen,  (2012).

\bibitem{MR2189672}
{\sc S.~Novo}, {\em Compressible {N}avier--{S}tokes model with inflow-outflow
  boundary conditions}, J. Math. Fluid Mech., 7 (2005), pp.~485--514.

\bibitem{MR2084891}
{\sc A.~Novotn\'{y} and I.~Stra\v{s}kraba}, {\em Introduction to the
  mathematical theory of compressible flow}, vol.~27 of Oxford Lecture Series
  in Mathematics and its Applications, Oxford University Press, Oxford, 2004.

\bibitem{peza2015}
{\sc C.~Perrin and E.~Zatorska}, {\em Free/congested two-phase model from weak
  solutions to multi-dimensional compressible {N}avier-{S}tokes equations},
  Communications in Partial Differential Equations, 40 (2015), pp.~1558--1589.

\bibitem{Song}
{\sc Y.~Song, E.~A. Mason, and R.~M. Stratt}, {\em Why does the
  {C}arnahan-{S}tarling equation work so well?}, J. Phys. Chem., 93 (1989),
  pp.~6916--6919.

\end{thebibliography}
\bibliographystyle{siam}

\end{document}